\documentclass[letter,11pt]{amsart}

\usepackage{amsmath,amsfonts,amssymb}
\usepackage{subfigure}
\usepackage{multirow}
\usepackage{lscape}
\usepackage{hyperref}
\usepackage{tikz}

\newtheorem{lem}{Lemma}[section]

\newtheorem{ejem}{Example}[section]

\newtheorem{rmk}{Remark}[section]

\usepackage{geometry}
\def\IE{\mathrm{IE}}
\def\A{\mathcal{A}}
\def\X{\mathcal{X}}

\def\IE{\mathrm{IE}}

\def\R{\mathbb{R}}
\def\Z{\mathbb{Z}}

\newcommand{\dsum}{\displaystyle\sum}

\newcommand\IEFLP{{\rm IEFLP}}

\allowdisplaybreaks

\let\origmaketitle\maketitle
\def\maketitle{
  \begingroup
  \def\uppercasenonmath##1{} 
  \let\MakeUppercase\relax 
  \origmaketitle
  \endgroup
}

\begin{document}

\title[Intra-facility equity in $p$-Facility Location Problems]{\Large Intra-facility equity in discrete and continuous $p$-facility location problems}

\author[V. Blanco, A. Mar\'in \MakeLowercase{and} J. Puerto]{{\large V\'ictor Blanco$^{a,b}$, Alfredo Mar\'in$^c$ and  Justo Puerto$^{d,e}$}\medskip\\
$^a$Institute of Mathematics (IMAG), Universidad de Granada\\
$^b$ Dpt. Quant. Methods for Economics \& Business, Universidad de Granada\\
$^c$ Dpt. Stats. \& Operational Research, Universidad de Murcia\\
$^d$ Institute of Mathematics (IMUS), Universidad de Sevilla\\
$^e$ Dpt. Stats. \& Operational Research, Universidad de Sevilla\\
}

\address{IMAG, Universidad de Granada, SPAIN.}
\email{vblanco@ugr.es}

\address{Department of Statistics and Operational Research, Universidad de Murcia, SPAIN.}
\email{amarin@um.es}

\address{IMUS, Universidad de Sevilla, SPAIN.}
\email{puerto@us.es}

\date{}

\begin{abstract}
We consider facility location problems with a new form of equity criterion. Demand points 
have preference order on the sites where the plants can be located. The goal is to find the
location of the facilities minimizing the envy felt by the demand points with 
respect to the rest of demand points allocated to the same plant. After defining 
this new envy criterion and the general framework based on it, we  provide formulations that  model this approach in both the discrete and the continuous framework. 
The problems are illustrated with examples and the computational tests reported show 
the potential and limits of each formulation on several types of instances. 
Although this article is mainly focused on the introduction, modelling  and formulation of this 
new concept of envy, some improvements for all the formulations presented are
developed, obtaining in some cases better solution times.
\end{abstract}

\keywords{Facility Location, Fairness, Equity, Mixed Integer Linear Programming}

\maketitle

\section{Introduction}

The notion of equity has played a very important role in Decision Theory and Operations Research over the years since it is one of the most popular criteria to judge fairness. It is mainly due to the fact that satisfaction is envy-free, i.e., a solution of a decision process such that every involved decision-maker likes its own solution at least as much as the one of any other agent. In particular, envy-freeness criterion has been used in many different problems related with fair resource allocation, fair queuing processes, fair auctions or pricing problems, among others (see e.g.\ 
\cite{Brams&Fishburn:2000,Chun:2006,DallAglio&Hill:2003,dominguez2022mixed, Haake&Rait&Su:2002, Moulin:2014,Ohseto:2005,Papai:2003,Reijnierse&Potters:1998,Shioura&Sun&Yang:2006,Webb:1999}).

One of the most analyzed areas in Operations Research is Facility Location, whose goal is to find the most appropriate positions for a set of services in order to satisfy the demand required by a set of users. Facility location problems are classified according to two main characteristics: the solution domain and the criteria used to evaluate the goodness of a solution. Concerning the solution domain, in case the placements of the services are to be chosen from a finite set of potential facilities, the problem is called a \textit{Discrete Facility Location problem} (DFLP) while if the position of the facilities are chosen from a continuous set, the problem is called a \textit{Continuous Facility Location problem} (CFLP). Apart from the practical applications of these two frameworks, the main difference between these two types of frameworks stems in the tools that are applied to solve the problems. While in 
 DFLP Integer Linear Programming is the most common tool used to solve, exactly, the problems, in  CFLP one resorts to the use of Convex analysis or Global Optimization tools due to the non-linear nature of the problems. The interested reader is referred to \cite{Laporte&Saldanha&Nickel:2019} for recent contributions in this area.
 
Concerning the evaluation criteria, the most popular measure to evaluate a given set of positions with respect to a set of customers is the sum of the overall allocation costs. This cost usually represents the access cost of customers to the facilities, and it is an accurate indicator of the efficiency of the logistic system. With such a measure one assumes that for the customers, the less costly the better, which is a natural assumption.  This criterion is the one used in the well-known $p$-median~\cite{Hakimi:1964} problem or the multifacility continuous location problem~\cite{Blanco&Puerto&ElHaj:2016}. However, other measures have also been proposed in the literature.

In this paper, we introduce an equity criterion that can be incorporated to make decisions on different facility location problems. Measuring the goodness of the solutions of location problems by means of the equity criterion is not new. However, still the body of literature analyzing facility location problems under equity lens is scarce, being  the maximum the most representative objective function in this area, giving rise to the well-known (discrete and continuous) $p$-center problems. In spite of that, some authors have already analyzed the use of criteria which promote equitable or fair solutions in some facility location problems.  

Savas \cite{Savas:1978} stated the insufficiency of efficiency and effectiveness measures
in location models for public facilities. 
Halpern and Maimon \cite{Halpern&Maimon:1981} considered a large number of tree networks in order to determine the agreement and disagreement of the solutions to location problems 
using the median, center, variance and Lorenz measure. 
Mulligan \cite{Mulligan:1991} designed a simple experiment consisting of locating a facility in an interval of the real straight line regarding three demand 
points. Apart from a comparative analysis of the optimal solutions for nine equality measures, he   also provided the standardized travel distance curves for them. 
Erkut \cite{Erkut:1993} proposed a general framework for quantifying inequality and presented some axioms for the appropriateness of the inequality measures. He also showed that only two of his considered measures --the coefficient of variation and the Gini coefficient-- hold both the scale-invariance property and the principle of transfers or Pigou-Dalton property. 
Berman and Kaplan \cite{Berman&Kaplan:1990} addressed the equity question using taxes.
Mar\'in et al.\ \cite{Marin&Nickel&Velten:2010} considered a general case of equity criterion for the Discrete Ordered Median Problem. 
Mesa et al.\ \cite{Mesa&Puerto&Tamir:2003} and Garfinkel et al.\  \cite{Garfinkel&Fernandez&Lowe:2006} addressed some algorithmic aspects of equity measures on network location and routing. Bertsimas et. al~\cite{Bertsimas&Farias&Trichakis:2012} proposed a fairness measure and incorporated it to resource allocation problems. More recently, Filippi et. al~\cite{Filippi&Guastttaroba&Speranza:2021} apply the conditional $\beta$-mean as a fairness measure for the bi-objective single-source capacitated facility location problem. Blanco and G\'azquez\ \cite{Blanco&Gazquez:2022} have introduced a new fairness measure   for the maximal covering location problem combining the approach in \cite{Bertsimas&Farias&Trichakis:2012} with OWA operators.
A review of the existing literature on equity measurement in Location Theory and a discussion on how to select an appropriate measure of equality was contained in the paper by Marsh and Schilling 
\cite{Marsh&Schilling:1994}. 
Also the equality objective literature was reviewed in the book by Eiselt \cite{Eiselt:1995} within a general discussion of objectives in Location Theory based on the physic concepts of pulling, pushing and balancing forces.

The envy criterion has also been considered in Location Theory as an equity criterion. In this framework envy is defined with respect to the revealed preference of each demand point for the sites of the potential serving facilities.
It is assumed that the users (demand points) elicit their preferences for the sites where the plants can be located. A positive envy appears between two users  when one of them is allocated to a plant strictly most preferred than the the one where it is allocated the other. The goal is to find the location of the facilities minimizing 
the total envy felt by the entire set of demand points. A limitation of this approach is that information is common knowledge. Thus, it is
assumed that the decision-maker has a previous complete knowledge of the preferences of all customers at the 
demand points or, alternatively, that all customers do not lie when they are asked about their preferences.
Espejo et al.\ \cite{Espejo&Marin&Puerto&RodriguezChia:2010} study a discrete facility location problem with 
an envy criterion. There, it is assumed that demand points feel envy with respect to all other better 
located demand points. Instead, in this paper we do not require that users have a complete knowledge of the envy felt by all the users. Rather than that, only partial knowledge of the points that are allocated to the same facility is assumed. The reason is that we assume that location points will only felt envy with respect to the users allocated to its same facility.  On the other hand, in \cite{Espejo&Marin&Puerto&RodriguezChia:2010} the envy is measured using only the 
preferences of the points,  while in our case the dissatisfaction of the user is given by a function that 
is independent of the preference function.

Following with the applications of the envy concept to Location Theory, one can find an adaptation of the envy  concept to the system of ambulances location in Chanta et al.\ \cite{Chanta&Mayorga&Kurz&McLay:2011} with the objective of
minimizing the sum of envies among all users with respect to an ordered set of operating stations. Also in
\cite{Chanta&Mayorga&McLay:2014}, envy is redefined as differences of customers’ satisfaction between users,  where satisfaction is measured by the survival probability of each user.

The model addressed in this paper fits well with the location of  services provided by public administrations where the planner (decision-maker) takes into account users preferences but at the same time wishes to offer the same service quality to all users allocated to the same service center. This policy ensures that when users meet at the service center no-one  has incentive to complaint on the basis of others users' quality of service. This may be applicable to outpatient consultations in public health systems where a  central authority allocates population areas to outpatient centers or to the location of civic centers offering service to population areas. In the continuous setting, the intra-envy location applies to determine the position of HPC servers and job allocation to them such that the communication/energy costs of all jobs allocated to the same node are to be similar to ensure fair comparison of the computational hardness of the jobs run in the same node (which are supposed to run under the same conditions)~\cite{Meng&McCauley&Kaplan&Leung&Coskun:2015}.

The paper is organized as follows. The intra-envy location problem in a general framework is 
introduced in Section \ref{general}. Applications to the continuous and the discrete cases, 
including formulations and improvements can be found in sections \ref{continuo} and 
\ref{discreto}, respectively. Some computational results that highlight the limits of the 
solution approaches presented in the paper can be seen in Section \ref{resultados}. 
Finally, some conclusions are drawn.

\section{The General Minimum  $p$-Intra-Envy Facility Location Problem} \label{general}

In this section we introduce the problem that we analyze and fix the notation for the remaining sections. We provide here a framework for the  $p$-Intra-Envy Facility Location Problem for general domains and cost functions.

Let $\A=\{a_1,\dots, a_n\}$ be a finite set of demand points in $\R^d$ indexed in set $N=\{1,\dots,n\}$.  Abusing of notation, throughout this paper we  refer to a demand point interchangeably by  $a_i$ or by the index $i$, for $i\in N$. Denote by $\X$ a (not necessarily finite) set of points also in $\R^d$ that represent the potential locations for a facility. In order to quantify the cost incurred by a demand point when it is allocated to a facility in $\X$, each demand point, $a_i$, is assumed to be endowed with a cost function $\Phi_i: \X \rightarrow \R_+$ that represents the cost of allocating the service demanded by the user $i$ from each of the potential facilities in $\X$. These cost functions can be induced by distances or by general functions representing the dissatisfaction of the users for the different potential facilities. 

We are also given $p \in \Z$ with $p\geq 1$, and we denote by $P=\{1, \ldots, p\}$ the index set of the facilities to be located. A $p$-Facility location problem consists of choosing $p$ facilities from $\X$ minimizing certain cost function to the demand points. It is usual to assume, as we also do here, that users are allocated to the less costly open facility. We also assume that in case of ties, demand points are allocated to the facilities producing the less global intra-envy.

In this paper, we will measure the goodness of a selected set of facilities by the overall envy  of the pairwise allocation costs between demand points allocated to the same facility. 
 Specifically, if $\mathbf{X} = \{X_1, \ldots, X_p\} \subset \X$ is a given set of facilities
and $j(\ell) := \arg\min_{j\in P} \Phi_\ell(X_j)$, i.e., the most preferred facility for the user, for each pair of demand points $i, k \in N$, we compute the intra-envy of $i$ for $j$ as:
$$
\IE_{ik}(\mathbf{X})  = \begin{cases}
 	\Phi_i(X_{j(i)})-\Phi_k(X_{j(k)}) & \mbox{if $\Phi_k(X_{j(i)})<\Phi_i(X_{j(k)})$ and $j(i)=j(k)$,}\\
   0 & \mbox{otherwise.}  
\end{cases}
$$

That is, if the less costly open facility for $i$ and $k$ coincides and the allocation cost of $k$ is smaller than the allocation cost for $i$, an intra-envy of $i$ for $k$ is incurred, equal to the difference between the allocation cost of $i$ and the allocation cost of $k$.

The goal of the  $p$-Intra-Envy Facility Location Problem 
($p$-\IEFLP, for short) is to chose $p$ out of the facilities from $\X$ minimizing the overall sum of the intra-envy of all pairs of demand points allocated to the same facility. Formally, the problem can be formulated as follows:
\begin{equation}\label{pIEFLP:1}\tag{$p$-IEFLP}
\min_{\mathbf{X} \subseteq \X:\atop |\mathbf{X}|=p} \dsum_{i\in N}\dsum_{k\in N} \IE_{ik}(\mathbf{X}).
\end{equation}

Observe that, avoiding duplicates and zeros in the expression being minimized above, the objective function can be equivalently rewritten as:

$$
\dsum_{i\in N}\dsum_{k\in N} \IE_{ik}(\mathbf{X}) = \dsum_{i =1}^{n-1} \dsum_{k=i+1 :\atop j(k)=j(i)}^n |\Phi_i(X_{j(i)})-\Phi_k(X_{j(k)})|.
$$
Using the usual allocation variables in facility location
\begin{eqnarray*}
 x_{ij} & = &
\begin{cases}
   1 & \mbox{if plant $j$ is the less costly open facility for $i$,} \\
   0 & \mbox{otherwise,}
\end{cases} \hspace{1cm} \forall i\in N, j\in P,
\end{eqnarray*}
and the variables $X_1, \ldots, X_p \in \X$ that represent 
the locations for the  facilities,
the problem can be stated as follows:
\begin{align}
\min &\dsum_{i \in N} \dsum_{k\in N:\atop k>i} \dsum_{j \in P} |\Phi_i(X_j)-\Phi_k(X_j)| x_{ij} x_{kj},\label{miemp:0}\\
\mbox{s.t.} & 
\dsum_{j \in P} x_{ij} =1,\ \ \forall i\in N,\label{miemp:1}\\
& \Phi_i(X_{j(i)}) \le \Phi_i(X_j),\ \ \forall i\in N, j \in P,\label{miemp:3}\\
& X_1, \ldots, X_p \in \X,\label{miemp:4}\\
& x_{ij} \in \{0,1\},\ \ \forall i \in N, j \in P. \nonumber
\end{align}
The objective function \eqref{miemp:0} accounts for intra-envy between pairs of points allocated to the same facilities. Constraints \eqref{miemp:1} assure that a single allocation is obtained for each demand point and \eqref{miemp:3} model the closest-assignment assumption. The set of constraints \eqref{miemp:4} indicates the domain of the positions of the facilities on a set of potential facilities. 

We analyze here the two main frameworks in facility location based on the nature of the set $\X$. On the one hand, we will study the \textit{continuous} case in which $\X=\R^d$ and, then, the facilities are allowed to be located in the whole decision  space. On the other hand, we will analyze the \textit{discrete} case in which $\X=\{b_1, \ldots, b_m\}$ is a finite set of potential locations for the facilities. 

Note that there are several differences when analyzing $p$-IEFLP under these two different frameworks:
\begin{enumerate}
\item In the discrete case, the possible costs between the demand points and the potential facilities can be obtained in a preprocessing phase since the possibilities for the values of $X_j$ are known and finite. Thus, one can compute a cost matrix 
 $ \Phi = (\Phi_i(b_j))\in \R^{n\times m}$ 
 that serves as input for the problem. In contrast, in the continuous case the costs can only be known when the coordinates of the facilities are computed and, then, their values have to be incorporated as decision variables  to the problem.
 Furthermore, in a continuous problem, the shape of the cost functions $\Phi_1, \ldots, \Phi_n$ has a significant impact in deriving a suitable mathematical programming formulation of the problem. 
\item The closest-assignment constraints \eqref{miemp:3} must be treated differently for  the continuous and the discrete case because of the knowledge of the cost values. For the discrete case, there are different approaches to incorporate linear constraints enforcing this requirement (see \cite{Espejo&Marin&Puerto&RodriguezChia:2010}).
 In the continuous case this requirement must also be ensured using different strategies. 
\item The absolute values in the objective function measuring the envy in case the costs are known, as in the discrete case, are constant values. In the continuous case  these values are unknown and part of the decision problem.
\end{enumerate}

In what follows, we illustrate the different situations when locating $p$ facilities with the different criteria, namely, the $p$-median, the $p$-envy (\cite{Espejo&Marin&Puerto&RodriguezChia:2010}) and the $p$-intra-envy
with the continuous and the discrete frameworks.

\begin{ejem}\label{ex:1}
Consider the six demand points in the real line $\A=\{1, 2, 4, 6, 10, 14\}$. For the discrete problem, we assume that each demand point is also a potential facility. We aim to locate $p=2$ facilities.

In the discrete setting, the dissatisfaction matrix (based on distances) can be prespecified as:
$$ \Phi=\begin{pmatrix}
0&1&3&5&9&13\\
1&0&2&4&8&12\\
3&2&0&2&6&10\\
5&4&2&0&4&8\\
9&8&6&4&0&4\\
13&12&10&8&4&0\\
\end{pmatrix} .$$

In Figure \ref{fig:ex1:0} we show the optimal solutions for the $p$-median, the 
$p$-envy and the $p$-intra-envy problems (from top to bottom). The open facilities are highlighted with gray colour circles.

\begin{figure}[h]
\begin{center}
\fbox{\begin{tikzpicture}[scale=1]

\def \margin {2} 

  \node[draw, circle] (A1) at (1,0) {\tiny $1$};
  \node[draw, circle, fill=gray] (A2) at (2,0) {\tiny $2$};
  \node[draw, circle] (A4) at (4,0) {\tiny $4$};
  \node[draw, circle](A6) at (6,0) {\tiny $6$};
  \node[draw, circle](A10) at (10,0) {\tiny $10$};
  \node[draw, circle, fill=gray](A14) at (14,0) {\tiny $14$};
  
  \path[->] (A1) edge node[above] {\tiny 1} (A2);
  \path[->] (A2) edge[out=135,in=45,looseness=5] node[above] {\tiny 0} (A2);
  \path[->] (A4) edge  node[above] {\tiny 2} (A2);
  \path[->] (A6) edge [out=270, in=270, ,looseness=0.35]  node[below] {\tiny 4} (A2);
 \path[->]  (A10) edge node[above] {\tiny 4} (A14);
 \path[->]  (A14) edge [out=135,in=45,looseness=5]   node[above] {\tiny 0} (A14);

 \node[draw, circle] (B1) at (1,-2) {\tiny $1$};
  \node[draw, circle] (B2) at (2,-2) {\tiny $2$};
  \node[draw, circle, fill=gray] (B4) at (4,-2) {\tiny $4$};
  \node[draw, circle](B6) at (6,-2) {\tiny $6$};
  \node[draw, circle, fill=gray](B10) at (10,-2) {\tiny $10$};
  \node[draw, circle](B14) at (14,-2) {\tiny $14$};
  
  \path[->] (B1) edge[out=270, in=270, ,looseness=0.35]  node[below] {\tiny 3} (B4);
  \path[->] (B2) edge node[above] {\tiny 2} (B4);
  \path[->] (B4) edge [out=135,in=45,looseness=5]   node[above] {\tiny 0} (B4);
  \path[->] (B6) edge  node[above] {\tiny 2} (B4);
 \path[->]  (B10) edge [out=135,in=45,looseness=5]  node[above] {\tiny 0} (B10);
 \path[->]  (B14) edge  node[above] {\tiny 4} (B10);

 \node[draw, circle] (C1) at (1,-4) {\tiny $1$};
  \node[draw, circle, fill=gray] (C2) at (2,-4) {\tiny $2$};
  \node[draw, circle] (C4) at (4,-4) {\tiny $4$};
  \node[draw, circle](C6) at (6,-4) {\tiny $6$};
  \node[draw, circle, fill=gray](C10) at (10,-4) {\tiny $10$};
  \node[draw, circle](C14) at (14,-4) {\tiny $14$};
  
  \path[->] (C1) edge node[above] {\tiny 1} (C2);
  \path[->] (C2) edge [out=135,in=45,looseness=5]  node[above] {\tiny 0} (C2);
  \path[->] (C4) edge  node[above] {\tiny 2} (C2);
  \path[->] (C6) edge  node[above] {\tiny 4} (C10);
 \path[->]  (C10) edge [out=135,in=45,looseness=5]  node[above] {\tiny 0} (C10);
 \path[->]  (C14) edge  node[above] {\tiny 4} (C10);
 
\end{tikzpicture}}
\end{center}
\caption{Solution of the discrete problems of Example \ref{ex:1} (from top to bottom: $p$-median, $p$-envy and 
$p$-intra-envy problems)\label{fig:ex1:0}}
\end{figure}
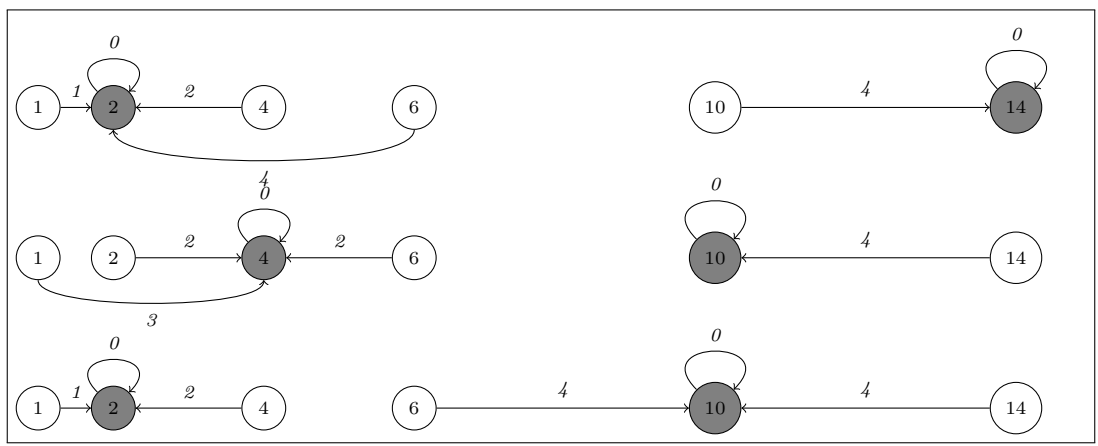

The optimal solution of the $p$-median problem on this instance, where the aim is to minimize  the total distance between plants and their allocated users, is obtained locating plants at sites $2$ and  $14$ and the optimal allocation pattern can be seen in the first picture of the figure.  There, customer located at $1$ is allocated to the facility in position $2$ at a distance of $1$ unit. Customer at $6$, which is allocated to the same plant, and whose distance to the plant is $4$, \emph{feels} envy from $1$ of $4-1=3$ units. However customer at position $10$ does not pay attention to customer at $1$ since it is not allocated to its same plant. Using the notation introduced above, the intra-envy matrix induced for facilities located at positions $2$ and $14$ is
$$
{\rm IE}(2,14) = \begin{pmatrix}
  0 & 1 & 0 & 0 & 0 & 0\\
  0 & 0 & 0 & 0 & 0 & 0\\
  1 & 2 & 0 & 0 & 0 & 0\\
  3 & 4 & 2 & 0 & 0 & 0\\
  0 & 0 & 0 & 0 & 0 & 4\\
  0 & 0 & 0 & 0 & 0 & 0\\
\end{pmatrix}.
$$
Here, each row represents the envy of a customer with respect to each of the other customers. The overall intra-envy is in this case $17$ units.

An optimal solution of the minimum envy problem previously studied in \cite{Espejo&Marin&Puerto&RodriguezChia:2010} for this
instance is shown in the second picture of Figure \ref{fig:ex1:0}. There, instead of computing the envy only with respect to the customers allocated to the same plant, the envy is calculated with respect to all the customers, no matter if the customers are allocated to the same facility or not. For this problem it is optimal to locate plants at $4$ and $10$, and the corresponding intra-envy matrix is in this case
$$
{\rm IE}(2,10) = \begin{pmatrix}
  0 & 1 & 3 & 1 & 0 & 0\\
  0 & 0 & 2 & 0 & 0 & 0\\
  0 & 0 & 0 & 0 & 0 & 0\\
  0 & 0 & 2 & 0 & 0 & 0\\
  0 & 0 & 0 & 0 & 0 & 0\\
  0 & 0 & 0 & 0 & 4 & 0\\
\end{pmatrix}.
$$
The overall intra-envy is now $13$ units.

In the third line of the figure we see the optimal solution to the discrete $p$-intra-envy location problem. The plants have been  located at nodes $2$ and $10$. The envy matrix is
$$
{\rm IE}(2,10)  =\begin{pmatrix}
  0 & 1 & 0 & 0 & 0 & 0\\
  0 & 0 & 0 & 0 & 0 & 0\\
  1 & 2 & 0 & 0 & 0 & 0\\
  0 & 0 & 0 & 0 & 4 & 0\\
  0 & 0 & 0 & 0 & 0 & 0\\
  0 & 0 & 0 & 0 & 4 & 0\\
\end{pmatrix}
$$
being the total intra-envy equal to $12$. Note that, for instance,  user $6$ which is at the same distance from $2$ than from $10$ feels a lower envy to its \emph{neighbours} when it is allocated to $10$ instead of being allocated to $2$, implying an overall smaller intra-envy.

\end{ejem}

\begin{ejem}
In Figure \ref{fig:ex1:1} we show the results for the $2$-facility Weber problem, the $2$-envy problem and the $2$-intra-envy problem on the plane (with cost function measure by the $\ell_1$-norm) for the set of $6$ demand points $\{(8,1), (1,13), (17,11), (18,15), (11,9), (19,7)\}$.

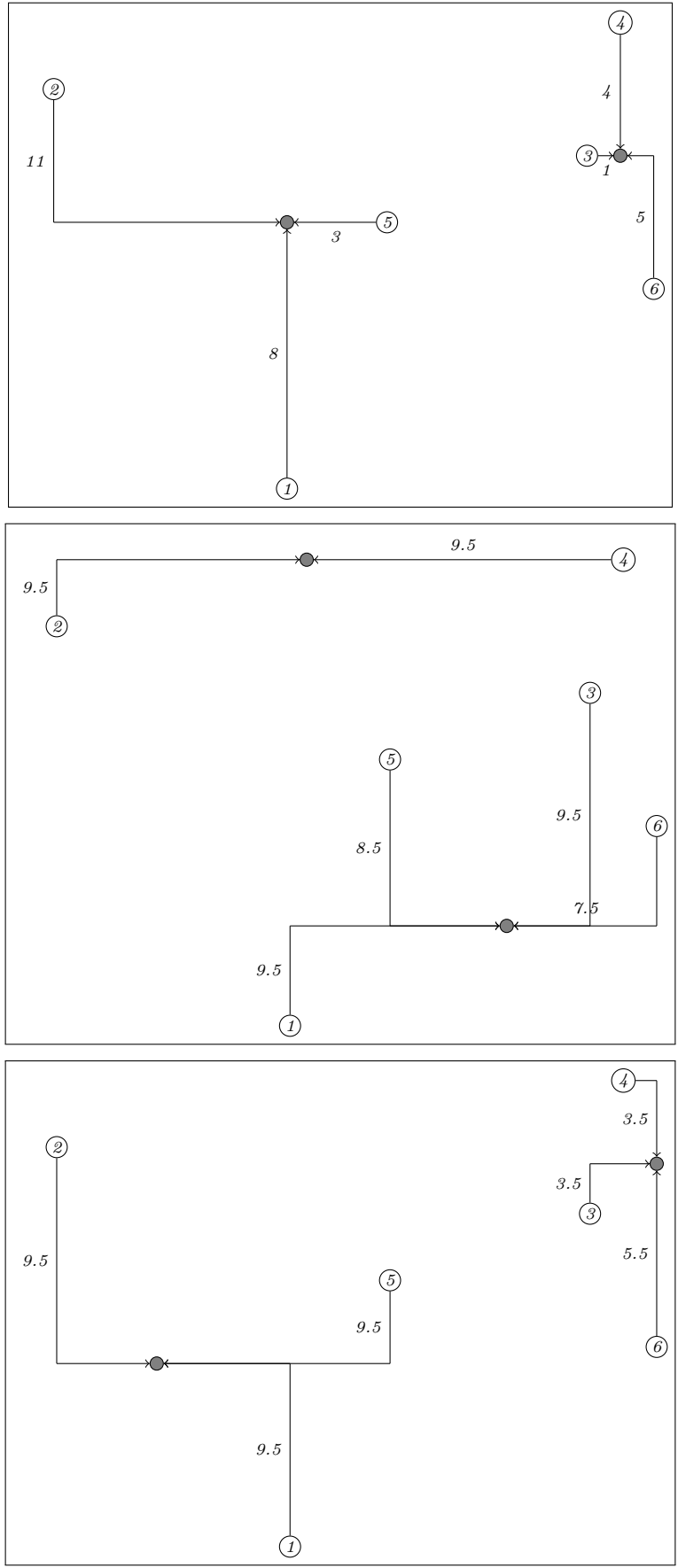
\begin{figure}
\begin{center}
\fbox{\begin{tikzpicture}[scale=0.5]

\def \margin {2} 

\node[draw, circle, inner sep=1pt](A1) at (8.00,1.00) {\tiny 1};
\node[draw, circle, inner sep=1pt](A2) at (1.00,13.00) {\tiny 2};
\node[draw, circle, inner sep=1pt](A3) at (17.00,11.00) {\tiny 3};
\node[draw, circle, inner sep=1pt](A4) at (18.00,15.00) {\tiny 4};
\node[draw, circle, inner sep=1pt](A5) at (11.00,9.00) {\tiny 5};
\node[draw, circle, inner sep=1pt](A6) at (19.00,7.00) {\tiny 6};
\node[draw, circle, fill=gray, inner sep=2pt](X1) at (8.00,9.00) {};
\node[draw, circle, fill=gray,  inner sep=2pt](X2) at (18.00,11.00) {};
\path[->] (A1) edge node[left] {\tiny 8} (X1);
\path (A2) edge node[left] {\tiny 11} (1.00,9.00);
\path[->] (1.00,9.00) edge node[above] {} (X1);
\path[->] (A5) edge node[below] {\tiny 3} (X1);
\path[->] (A3) edge node[below] {\tiny 1} (X2);
\path[->] (A4) edge node[left] {\tiny 4} (X2);
\path (A6) edge node[left] {\tiny 5} (19.00,11.00);
\path[->] (19.00,11.00) edge node[above] {} (X2);
\end{tikzpicture}}\\

\vspace*{0.2cm}

\fbox{\begin{tikzpicture}[scale=0.5]

\def \margin {3} 

\node[draw, circle, inner sep=1pt](A1) at (8.00,1.00) {\tiny 1};
\node[draw, circle, inner sep=1pt](A2) at (1.00,13.00) {\tiny 2};
\node[draw, circle, inner sep=1pt](A3) at (17.00,11.00) {\tiny 3};
\node[draw, circle, inner sep=1pt](A4) at (18.00,15.00) {\tiny 4};
\node[draw, circle, inner sep=1pt](A5) at (11.00,9.00) {\tiny 5};
\node[draw, circle, inner sep=1pt](A6) at (19.00,7.00) {\tiny 6};
\node[draw, circle, fill=gray, inner sep=2pt](X1) at (8.50,15.00) {};
\node[draw, circle, fill=gray, inner sep=2pt](X2) at (14.50,4.00) {};
\path (A2) edge node[left] {\tiny 9.5} (1.00,15.00);
\path[->] (1.00,15.00) edge node[above] {} (X1);
\path[->] (A4) edge node[above] {\tiny 9.5} (X1);
\path (A1) edge node[left] {\tiny 9.5} (8.00,4.00);
\path[->] (8.00,4.00) edge node[above] {} (X2);
\path (A3) edge node[left] {\tiny 9.5} (17.00,4.00);
\path[->] (17.00,4.00) edge node[above] {} (X2);
\path (A5) edge node[left] {\tiny 8.5} (11.00,4.00);
\path[->] (11.00,4.00) edge node[above] {} (X2);
\path (A6) edge node[left] {} (19.00,4.00);
\path[->] (19.00,4.00) edge node[above] {\tiny 7.5} (X2);
 
\end{tikzpicture}}\\

\vspace*{0.2cm}

\fbox{\begin{tikzpicture}[scale=0.5]

\def \margin {2} 

\node[draw, circle, inner sep=1pt](A1) at (8.00,1.00) {\tiny 1};
\node[draw, circle, inner sep=1pt](A2) at (1.00,13.00) {\tiny 2};
\node[draw, circle, inner sep=1pt](A3) at (17.00,11.00) {\tiny 3};
\node[draw, circle, inner sep=1pt](A4) at (18.00,15.00) {\tiny 4};
\node[draw, circle, inner sep=1pt](A5) at (11.00,9.00) {\tiny 5};
\node[draw, circle, inner sep=1pt](A6) at (19.00,7.00) {\tiny 6};
\node[draw, circle, fill=gray, inner sep=2pt](X1) at (4.00,6.50) {};
\node[draw, circle, fill=gray, inner sep=2pt](X2) at (19.00,12.50) {};
\path (A1) edge node[left] {\tiny 9.5} (8.00,6.50);
\path[->] (8.00,6.50) edge node[above] {} (X1);
\path (A2) edge node[left] {\tiny 9.5} (1.00,6.50);
\path[->] (1.00,6.50) edge node[above] {} (X1);
\path (A5) edge node[left] {\tiny 9.5} (11.00,6.50);
\path[->] (11.00,6.50) edge node[above] {} (X1);
\path (A3) edge node[left] {\tiny 3.5} (17.00,12.50);
\path[->] (17.00,12.50) edge node[above] {} (X2);
\path (A4) edge node[left] {} (19.00,15.0);
\path[->] (19.00,15.0) edge node[left] {\tiny 3.5} (X2);
\path[->] (A6) edge node[left] {\tiny 5.5} (X2);
 
\end{tikzpicture}}

\end{center}
\caption{Solution of the continuous problems of Example \ref{ex:1} (from top to bottom: $p$-median, $p$-envy and 
$p$-intra-envy problems)\label{fig:ex1:1}}
\end{figure}

The solution of the $2$-facility Weber problem is shown in the top picture. The facilities to be located are $X_1=(8,9)$ and $X_2=(18,11)$. The intra-envy matrix is in this case
$$
{\rm IE}\left((8, 9), (18,11)\right)  =\begin{pmatrix} 
0 & 0 & 0 & 0 & 5 & 0\\
3 & 0 & 0 & 0 & 8 & 0\\
0 & 0 & 0 & 0 & 0 & 0\\
0 & 0 & 3 & 0 & 0 & 0\\
0 & 0 & 0 & 0 & 0 & 0\\
0 & 0 & 4 & 1 & 0 & 0
\end{pmatrix}
$$
with overall intra-envy equal to $24$. 

In the second picture we show the result of solving a modified version of the problem analyzed in \cite{Espejo&Marin&Puerto&RodriguezChia:2010} for the continuous case. The optimal facilities are located at positions $X_1=(8.5,15)$ and $X_2= (14.5,4)$. The intra-envy matrix in this case is
$$
{\rm IE}\left((8.5, 15), (14.5, 4)\right)  =\begin{pmatrix}
  0 & 0& 0& 0 & 1 & 2\\
  0 & 0& 0& 0 & 0 & 0\\
  0 & 0& 0& 0 & 1 & 2\\
  0 & 0& 0& 0 & 0 & 0\\
  0 & 0& 0& 0 & 0 & 1\\
  0 & 0& 0& 0 & 0 & 0
\end{pmatrix}
$$
with an overall intra-envy of $7$. In this solution, all the demand points  wish to be positioned at the same distance with respect to their closest facilities. Specifically, in all pairwise comparisons between the distances, a maximum difference of 2 units is observed.

Finally, in the bottom picture of Figure \ref{fig:ex1:1} we show the solution of the $2$-intra-envy problem in the continuous case. The two facilities are at $X_1=(4, 6.5)$, and $X_2=(19,12)$, being the intra-envy matrix
$$
{\rm IE}\left((4, 6.5), (19,12)\right)  =\begin{pmatrix}
  0 & 0 & 0 & 0 & 0 & 0\\
  0 & 0 & 0 & 0 & 0 & 0\\
  0 & 0 & 0 & 0 & 0 & 0\\
  0 & 0 & 0 & 0 & 0 & 0\\
  0 & 0 & 0 & 0 & 0 & 0\\
  0 & 0 & 2 & 2 & 0 & 0\\
\end{pmatrix}
$$
with an overall intra-envy of $4$ units. The envy is computed only comparing pairs of demand points allocated to the same facility. The ideal solution is that in which all the points allocated to a facility are at the same distance to it. In the  optimal solution, points numbered as $1$, $2$ and $5$  are allocated to the same facility at the same distance, and then the envy among those points is zero.
\end{ejem}

\section{The Continuous Minimum $p$-Intra-Envy Facility Location Problem}\label{continuo}

In this section we will study the $p$-\IEFLP\ assuming that the $p$ facilities to be located are allowed to be positioned in the whole space. In this case, we assume that the preference function of user $i$ with respect to a facility located at $X\in \R^d$ is given by the $\ell_1$-norm based distance in $\R^d$, i.e.,
$$
\Phi_i(X) =\| a_i- X\|_1 = \dsum_{l=1}^d |a_{il} - X_l|.
$$

We provide three alternative mathematical programming formulations for the problem. The first formulation is based on representing the envy as pairwise differences of distances.
 In contrast, in the second and third formulations, we exploit the structure of the intra-envy as an ordered median function of the distances. In both formulations, in order to adequately represent the distance in terms of the variables defining the positions of the facilities (which are part of the decision), we use the following sets of decision variables:
$$
X_{j\ell}: \text{$\ell$-th coordinate of the $j$-th facility, } \forall j\in P, \ell=1, \ldots, d.
$$
$$
\phi_{ij}=\Phi_i(X_j) = \|a_i-X_j\|_1,\ \ \forall i\in N, j \in P.
$$
We also use the allocation variables $x_{ij}$ described in the previous section.

\subsection*{First formulation for the Continuous $p$-\IEFLP}

In the first formulation, in addition, we use the following variables to model the intra-envy between two customers
$i$ and $k$ for the set of $p$ facilities
 $\mathbf{X}=\{X_1, \ldots, X_p\}$:
$$
\theta_{ik} = \IE_{ik}(\mathbf{X}), \ \ \forall i, k \in N (k>i).
$$

With the above sets of variables, the Continuous $p$-\IEFLP\ problem can be formulated as the following mathematical programming problem that we denote as (${\rm M}_1^C$):

\begin{align}
\min \; & \sum_{i\in N} \sum_{k\in N:\atop k>i} \theta_{ik} \label{cmiemp:obj} \\
\mbox{s.t. } &  
\sum_{j=1}^p x_{ij} =1,\ \ \forall i\in N,\label{cmiemp:x}\\
&\theta_{ik} \ge \phi_{ij}-\phi_{kj}- U (2-x_{ij}-x_{kj}),\ \ \forall i<k \in N, j \in P, \label{cmiemp:theta1}\\
&\theta_{ik} \ge \phi_{kj}-\phi_{ij}- U(2-x_{ij}-x_{kj}),\ \ \forall i<k \in N, j \in P, \label{cmiemp:theta2}\\
& \phi_{ij} \le \phi_{il}+ U (1-x_{ij}),\ \ \forall i \in N, j\neq \ell \in P,\label{cmiemp:close}\\
& \phi_{ij} = \Phi_i(X_j),\ \ \forall i \in N, j \in P,\label{cmpiemp:d}\\
& X_1, \ldots, X_p \in \R^d, \\
& \theta_{ik} \ge 0,\ \ \forall i,k\in N, k>i,\\
& x_{ij}\in \{0,1\},\ \ \forall i\in N, j \in P.\label{cmiemp:f}
\end{align}


\noindent where $U$ is a big enough constant. The reader may note that different $U's$ could be estimated for the different families of constraints, namely \eqref{cmiemp:theta1}-\eqref{cmiemp:theta2}  and \eqref{cmiemp:close}, although we keep it simpler for the sake of presentation.

In this model the objective function \eqref{cmiemp:obj} accounts for the intra-envy 
felt by every client with respect to all other clients. Constraints  \eqref{cmiemp:theta1}  and \eqref{cmiemp:theta2} allow to adequately represent the envy between users $i$ and $k$, i.e., either $\IE_{ik}(X)=\Phi_i(X_j)-\Phi_k(X_j)$ or $\IE_{ki}(X)=\Phi_k(X_j)-\Phi_i(X_j)$ in case $i$ and  $k$ are allocated to facility $j$. Constraints \eqref{cmiemp:close} assure the closest-assignment assumption. Constraints \eqref{cmpiemp:d} are the representation of the $\ell_1$ distances between customers and the facilities. For $i\in N$ and $j\in P$,  the constraint can be linearly modeled as follows:

\begin{align}
&w_{ij\ell} \leq a_{i\ell} - X_{j\ell}  + U (1-\xi_{ij\ell}),\ \ \forall  \ell=1, \ldots, d \label{mfall1:w1}\\
& w_{ij\ell} \geq a_{i\ell} - X_{j\ell},\ \ \forall \ell=1, \ldots, d,\label{mfall1:w2}\\
& w_{ij\ell} \leq -a_{i\ell} + X_{j\ell} + U \xi_{ij\ell},\ \  \forall \ell=1, \ldots, d,\label{mfall1:w3}\\
& w_{ij\ell} \geq -a_{i\ell} + X_{j\ell},\ \  \forall \ell=1, \ldots, d,\label{mfall1:w4}\\
& \phi_{ij} = \dsum_{\ell=1}^d w_{ij\ell}, \label{mfall1:d}\\
&w _{ij\ell} \geq 0,\ \ \forall  \ell=1, \ldots, d,\nonumber\\
&\xi _{ij\ell} \in \{0,1\},\ \ \forall \ell=1, \ldots, d\nonumber
\end{align}
where two sets of additional auxiliary variables are used:  $w_{ij\ell}=|a_{i\ell}-X_{j\ell}|$, and $\xi_{ij\ell}$ that takes values $1$ in case $w_{ij\ell}=a_{i\ell}-
X_{j\ell} \geq 0$ and zero otherwise. Constraints \eqref{mfall1:w2} and \eqref{mfall1:w4} ensure that $w_{ij\ell} \geq \max\{a_{i\ell}-X{j\ell}, -a_{i\ell}+X_{j\ell}\}$ while constraints \eqref{mfall1:w1} and \eqref{mfall1:w3}  assure that $w_{ij\ell} \leq \max\{a_{i\ell}-X_{j\ell}, -a_{i\ell}+X_{j\ell}\}$, defining adequately the absolute value $|a_{i\ell}-X_{j\ell}|$. Constraints \eqref{mfall1:d} define the $\ell_1$ distance by means of the sum of the absolute values of the differences at all the coordinates between the demand point and the facility.  Note that with this representation of the distance, $\|X_j-a_i\|_1$ can be expressed by the sum of the adequate $w$-variables \eqref{mfall1:d}.

\subsection*{A $k$-sum based formulation for the Continuous $p$-\IEFLP }
The second formulation that we propose is based on the following observation.

\begin{lem}
Let $\phi_{kj} =\Phi_i(X_k)$ if $a_i$ is allocated to $X_k$ and zero otherwise,  $\phi_{(k)j}$ the $k$-th largest distance in the sequence of distances from all the demand points to $X_j$, and  $k_j$ the number of points allocated to the $j$th facility.
The intra-envy function can be written as:
$$
\dsum_{i\in N}\dsum_{k\in N} \IE_{ik}(X_1, \ldots, X_p) = \dsum_{j\in P} \dsum_{k\in N} (k_j-2k+1) \phi_{(k)j}.
$$
\end{lem}
\begin{proof}
Let $X_1, \ldots, X_p \in \R^d$ be the chosen facilities, and $\phi_{\cdot j} = (\phi_{1j}, \ldots, \phi_{nj})$ the allocation costs of all the demand points to $X_j$ (assuming that the allocation cost of a demand point non allocated to $X_j$ is zero). Sorting $\phi_j$ in non  increasing order results in the vector $(\phi_{(1)j}, \phi_{(2)j}, \ldots, \phi_{(k_j)j}, 0, \dots, 0)$ with $\phi_{(1)j} \geq \phi_{(2)j}\geq \cdots \geq \phi_{(k_j)j} \geq 0$.

With this notation, the intra-envy of the demand point sorted in the $k$th  position for facility $j\in P$ can be computed as:
\begin{eqnarray*}
\begin{split}
\dsum_{i=k+1}^{k_j} (\phi_{(k)j} - \phi_{(i)j}) &= (k_j-k) \phi_{(k)j} - \dsum_{i=k+1}^{k_j} \phi_{(i)j}\\
&=  (k_j-k) \phi_{(k)j} - \phi_{(k+1)j} - \cdots - \phi_{(k_j)j}.
\end{split}
\end{eqnarray*}
Adding up the above expression for all $k$ we get that the intra-envy for the $j$th facility can be written as:
\begin{eqnarray*}
\begin{split}
\dsum_{k=1}^{k_j-1}\dsum_{i=k+1}^{k_j} (\phi_{(k)j} - \phi_{(i)j}) &= \dsum_{k=1}^{k_j-1} \Big((k_j-k) \phi_{(k)j} - \phi_{(k+1)}^j - \cdots - \phi_{(k_j)j}\Big)\\
&= (k_j-1) \phi_{(1)}^j + (k_j-3) \phi_{(2)}^j + \cdots + (1-k_j) \phi_{(k_j)j}\\
&= \dsum_{k=1}^{k_j} (k_j-2k+1) \phi_{(k)j} = \dsum_{k\in N} (k_j-2k+1) \phi_{(k)j}.
\end{split}
\end{eqnarray*}
\end{proof}

The above result identifies the objective function of $p$-\IEFLP\ as an ordered median function of the allocation costs given by $\Phi$ (see \cite{Mesa&Puerto&Tamir:2003}). This type of functions has been widely studied in Location Science 
(see e.g., \cite{Marin&Nickel&Puerto:2009}) and several representations are possible to embed these 
sortings in a mathematical programming formulation.
 One of the most popular representations is through the so-called $k$-sums which is based on expressing the ordered median function as a weighted sum of $k$-sums $S_k(\phi_{\cdot j})=\dsum_{\ell=1}^k \phi_{(\ell)j}$.
\begin{lem}
Let $j\in P$, $X_j \in \R^d$ and $\phi_{\cdot j} = (\phi_{1j}, \ldots, \phi_{nj})$ the allocation costs of all the demand points to $X_j$. Then:
$$
\dsum_{k\in N} (k_j-2k+1) \phi_{(k)j} = 2 \dsum_{k\in N: k<k_j}  \dsum_{\ell=1}^k \phi_{(\ell)j} - (2n+1-k_j) \dsum_{\ell=1}^{n} \phi_{(\ell)j}.
$$
\end{lem}
\begin{proof}
Observe that defining:
$$
\Delta_k^j = \begin{cases}
2 & \mbox{if $k<k_j$},\\
1- k_j & \mbox{if $k=k_j$},\\
0 & \mbox{otherwise,}
\end{cases}
$$we get that
\begin{eqnarray*}
\begin{split}
\dsum_{k\in N} (k_j-2k+1) \phi_{(k)j} &=\dsum_{k\in N} \Delta_k^j S_k(\phi_{\cdot j})=  \dsum_{k\in N} \Delta_k^j \dsum_{\ell=1}^k \phi_{(\ell)j}\\
&= 
2 \dsum_{k\in N: k<k_j}  S_k(\mathbf{\phi}^j)+ (1-k_j) S_{k_j}(\mathbf{\phi}^j)\\
&= 2 \dsum_{k\in N}  S_k(\mathbf{\phi}^j) - 2(n-k_j+1)\dsum_{\ell=1}^{n} \phi_{\ell j}+ (1-k_j)\dsum_{\ell=1}^{n} \phi_{\ell j} \\
&= 2 \dsum_{k\in N}  S_k(\phi_{\cdot j}) - (2n+1-k_j)\dsum_{\ell=1}^{n} \phi_{\ell j} 
\end{split}
\end{eqnarray*}
since $\phi_{(k)j} =0$ for all $k>k_j$ and then $S_{k}(\phi_{\cdot j})=S_{k_j}(\phi^j)= \dsum_{l=1}^n \phi_{lj}$. In the last equation the expression $\dsum_{\ell=1}^{n} \phi_{\ell j}$ is added up and removed to aggregate in the first addend all the $k$-sums.
\end{proof}

With the above results, the objective function of $p$-\IEFLP\ can be rewritten as:
$$
\dsum_{i\in N}\dsum_{k\in N} \IE_{ik}(\mathbf{X})
 = 2 \dsum_{j\in P} \dsum_{k\in N}  \dsum_{\ell=1}^k \phi_{(\ell) j}  - \dsum_{j\in P}  (2n+1-k_j) \dsum_{\ell=1}^{n} \phi_{\ell j}.
$$

Different representations of $k$-sums are possible when they are incorporated to optimization problems (see e.g., \cite{Blanco&Puerto&ElHaj:2014,MarinPoncePuerto20,Ogryczak&Tamir:2003,PuertoChiaTamir17}).
Specifically, we use the one proposed in \cite{Blanco&Puerto&ElHaj:2014} to derive the following mathematical formulation for the continuous $p$-EIFLP\ where, additionally to the above mentioned variables, we use the set of auxiliary variables $\alpha_{ij} = \sum_{\ell\in N} \phi_{\ell j} x_{ij}$ in order to represent the expression $k_j \dsum_{\ell=1}^{n} \phi_{\ell j} = \Big(\dsum_{i\in N} x_{ij}\Big) \dsum_{\ell\in N} \phi_{\ell j} = \dsum_{i, \ell \in N} x_{ij} \phi_{\ell j}$ in the objective function. The following formulation for the problem was denoted by ${\rm (M_2)}^C$.

\begin{align}
\min &\;\; 2 \dsum_{j\in P} \dsum_{k\in N} \left(\dsum_{\ell \in N} u_{k\ell j} + \dsum_{i\in N} v_{kij}\right) - (2n+1)\dsum_{j\in P}\dsum_{i \in N} \phi_{ij} + \dsum_{j\in P}\dsum_{i\in N} \alpha_{ij} \label{omp:obj}\\
\mbox{s.t. } & \dsum_{j\in P} x_{ij}=1\ \ \forall i\in N,\label{omp:1}\\
&w_{ij\ell} \leq a_{i\ell} - X_{j\ell}  + U (1-\xi_{ij\ell})\ \ \forall i\in N, j\in P,  \ell=1, \ldots, d \label{omp:2}\\
& w_{ij\ell} \geq a_{i\ell} - X_{j\ell}\ \ \forall i\in N, j\in P,\ell=1, \ldots, d,\label{omp:3}\\
& w_{ij\ell} \leq -a_{i\ell} + X_{j\ell} + U \xi_{ij\ell}\ \  \forall i\in N, j\in P, \ell=1, \ldots, d\label{omp:4}\\
& w_{ij\ell} \geq -a_{i\ell} + X_{j\ell}\ \  \forall i\in N, j\in P, \ell=1, \ldots, d,\label{omp:5}\\
& \phi_{ij} \leq \dsum_{l=1}^d w_{ijl} +  U(1-x_{ij})\ \ \forall i\in N, j\in P,\label{omp:6}\\
& \phi_{ij} \geq \dsum_{l=1}^d w_{ijl} - U(1-x_{ij})\ \ \forall i\in N, j\in P,\label{omp:7}\\ 
& \phi_{ij} \leq U x_{ij}\ \ \forall i\in N, j\in P,\label{omp:8}\\
& \alpha_{ij} \geq \dsum_{\ell\in N}\phi_{\ell j} - U (1-x_{ij})\ \ \forall i, \ell\in N, j\in P,\label{omp:14}\\
& u_{k\ell j} + v_{kij} \geq \phi_{ij}\ \ \forall i, k\in N, j\in P,\label{omp:10}\\
& u_{kij}, v_{kij}\geq 0\ \ \forall k, i \in N, j\in P,\\
& \phi_{ij} \geq 0, x_{ij} \in \{0,1\}\ \ \forall i\in N, j\in P,\\
& \alpha_{ij} \geq 0\ \ \forall i\in N, j\in P.\label{omp:f}
\end{align}

The objective function represents the overall $p$-intra-envy 
 as detailed in the above comments. Constraints \eqref{omp:1} are the single-allocation constraints. Contraints \eqref{omp:2}-\eqref{omp:5} ensure the correct definition of the absolute values $w_{ij\ell} = |a_{i\ell}-X_{j\ell}|$ required to derive the $\ell_1$ distances between demand points and the facilities. Constraints \eqref{omp:6} and \eqref{omp:7} assure that, in case the demand point $a_i$ is allocated to facility $j$, then, the cost of allocating such a point to that facility, $\phi_{ij}$, is defined as the $\ell_1$-norm based distance between $a_i$ and $X_j$. Otherwise, by constraints \eqref{omp:8} the cost $\phi_{ij}$ is fixed to zero. 
 Constraints \eqref{omp:14}
 (and the minimization of the $\alpha$-variables) ensure the correct definition of the $\alpha$-variables. Finally, contraints \eqref{omp:10} allow computing adequately the $k$-sums. 

\begin{rmk}
Apart from the $k$-sum representation applied above based on \cite{Blanco&Puerto&ElHaj:2014}, other representations are possible. Specifically, in \cite{Ogryczak&Tamir:2003} the authors provide an alternative formulation which has been widely used 
 in the literature. There, it is proved that
\begin{align*}
&S_k(\phi_{\cdot j}) = & \min &\;\;  k t_{kj} + \dsum_{i\in N} w_{kij}&\\
&&\hbox{s.t. } &t_{kj} +  w_{kij} \geq \phi_{ij}\ \ \forall i\in N&\\
&&& t_{kj}\geq 0,\\
&&& w_{kij} \geq 0\ | \forall i\in N.&
\end{align*}
This representation can be embedded in the following formulation that we call ${\rm (M_3)}^C$:
\begin{align}
\min &\;\; 2 \dsum_{j\in P} \dsum_{k\in N}  \Big(k t_{kj} + \dsum_{i\in N} w_{kij}\Big) - (2n+1)\dsum_{j\in P}\dsum_{i \in N} \phi_{ij} + \dsum_{j\in P}\dsum_{i\in N} \alpha_{ij} \label{ot:obj}\\
\hbox{s.t. } & \eqref{omp:1}-\eqref{omp:14}\\
& t_{kj} +  w_{kij} \geq \phi_{ij}\ \ \forall i, k\in N, j\in P&\\
& \phi_{ij} \geq 0, x_{ij} \in \{0,1\}\ \ \forall i\in N, j\in P,\\
& t_{kj}\geq 0\ \ \forall k\in N, j\in P,\\
& w_{kij} \geq 0\ \ \forall i\in N, k\in N, j\in P,\\
& \alpha_{ij} \geq 0\ \ \forall i\in N, j\in P.\label{ot:f}
\end{align}
\end{rmk}

\section{The Discrete Minimum $p$-Intra-Envy Facility Location Problem} \label{discreto}

In this section we analyze the case in which the set of potential positions for the facilities is a finite set, i.e., $\X=\{b_1, \ldots, b_m\} \subseteq \R^d$. We denote by $M=\{1, \ldots, m\}$ the index set for the potential facilities. We assume that  $1\le p\le m-1$ plants have to be located.

In this situation, the distances/costs between the users and all the potential sites can be computed in a preprocessing phase. We denote by $C=(\phi_{ij})_{n\times m}$ a costs matrix, 
 where $\phi_{ij} = \Phi_i(b_j)$ is the measure of the dissatisfaction user $i$ will feel if he is allocated to site $j$.

The problem in this case simplifies choosing
 $p$ facilities out of the $m$ potential facilities minimizing the overall envy of the customers. Thus, apart from the $x$-variables indicating the allocation of users to plants and the $\theta$-variables used to model the envy between customers, already defined in the previous sections, we use the following set of variables:
$$
 y_j  =
\begin{cases}
   1 & \mbox{if a plant is located at site $j$,} \\
   0 & \mbox{otherwise,}
\end{cases} \quad \forall j\in M.
$$

With the above notation, the Discrete Minimum $p$-Intra-Envy Facility Location Problem 
 can be formulated as the following mixed integer linear programming problem that we denote by (${\rm M}_1^D$):

\begin{align}
 \min &\;\; \sum_{i=1}^{n-1} \sum_{k=i+1}^n \theta_{ik} \label{pm:0}\\
\mbox{s.t.} & \sum_{j\in M} y_j = p, \ \ \label{pm:1}\\
& \sum_{j\in M} x_{ij} =1, \ \ \forall i\in N, \label{pm:2}\\
& x_{ij}\le y_j, \ \ \forall i\in N,\ j\in M, \label{pm:3}\\
&y_j + \sum_{\ell \in M:\atop \phi_{ij}<\phi_{i\ell}} x_{i\ell} \le 1, \ \ \forall i\in N,\ j\in M, \label{cac}\\
 & \theta_{ik} \ge |\phi_{ij}-\phi_{kj}|(x_{ij} + x_{kj} - 1), \ \ \forall i<k\in N,\ j\in M,  \label{teta1}\\
& y_i\in \{ 0,1\}, \ \ \forall j\in M, \\
&x_{ij}\in \{ 0,1\}, \ \ \forall i\in N,\ j\in M.\label{pm:f}
\end{align}

Constraints \eqref{pm:1}-\eqref{pm:3} are the classical $p$-median constraints that assure that $p$ services are open, each client is allocated to a single plant and customers are allowed to be allocated only to open plants. 
 Constraints \eqref{cac} are the closest assignment constraints. 
They avoid allocating a customer to plants which are less desired than others that are open.  Constraints \eqref{teta1} ensure the adequate definition of the envy variables $\theta$. Note that $\theta_{ik}$ takes value $|\phi_{ij}-\phi_{ik}|$ in case
$i$ and $k$ are allocated to a common plant $j$, i.e, whenever $x_{ij} \cdot x_{kj} = 1$ (term which is linearized in the constraint as $x_{ij}+x_{kj}-1$). These 
constraints can be strengthened by the following ones:
\begin{equation} \label{tetay}
  \theta_{ik} \ge |\phi_{ij}-\phi_{kj}|(x_{ij} + x_{kj} - y_j), \ \ \ \ \ \forall i<k\in N,\ j\in M, 
\end{equation}
since $y_j$ will take value 1 whenever $x_{ij}x_{kj}$ take value 1.

The above formulation is based on the classical variables and constraints of discrete location problems. A reduced formulation can be derived using only the $y$-variables, with the following formulation that we call (${\rm F1}^D$):

\begin{align}
\min &\;\; \sum_{i=1}^{n-1} \sum_{k=i+1}^n \theta_{ik} \label{m2:obj} \\
 \mbox{s.t.} & \sum_{j\in M} y_j = p, \\
 & \theta_{ik} \ge |\phi_{ij}-\phi_{kj}|
  (y_j-\sum_{\ell\in M:\atop {\phi_{i\ell}<\phi_{ij}\hbox{\ \tiny or}\atop \phi_{k\ell}<\phi_{kj}}} y_\ell), \ \ \forall i<k\in N,\ j\in M,  \label{teta2}\\
  & \theta_{ik} \geq 0, \forall i<k \in N,\\
  & y_{j} \in \{0,1\}, \forall j \in M.\label{m2:f}
\end{align}

Constraints \eqref{teta2} make $\theta_{ik}$ to take value $|\phi_{ij}-\phi_{kj}|$ when (i) plant $j$ 
is opened, and (ii) no other plant $\ell$ is opened if it is closer to $i$ or $k$ than plant $j$. This 
means that the closest plant (or one of the closest plants in case of tie) to both, $i$ and $k$, is
$j$, and therefore $i$ and $k$ will be both allocated to $j$ and the envy of this allocation will be  (at least, in case of tie) $|\phi_{ij}-\phi_{kj}|$.

The following set of valid inequalities can be used to tight the formulation:
$$\theta_{ik}\ge \sum_{j\in J} |\phi_{ij}-\phi_{kj}|(y_j - \dsum_{{\ell \in M :\atop \phi_{i\ell}<\phi_{ij}\hbox{\tiny or }} \atop \phi_{k\ell}<\phi_{kj}} y_\ell), \quad \forall J\subseteq M.$$

These inequalities were incorporated to (${\rm F1}^D$) sequentially in the branch-and-bound tree by separating them with the following strategy. Let $\bar y\in [0,1]^m$ and $\bar \theta \in \R_+^{n\times n}$ be a feasible relaxed solution. 
 For each $i, k \in N$, let $J_{ik} = \{j\in M: \bar y_j > \dsum_{{\ell \in M :\atop \phi_{i\ell}<\phi_{ij}\hbox{\tiny or }} \atop \phi_{k\ell}<\phi_{kj}} \bar y_\ell \}$ and $\rho_{ik} =\dsum_{j\in J_{ik}} |\phi_{ij}-\phi_{kj}|(\bar y_j - \sum_{{\ell \in M :\atop \phi_{i\ell}<\phi_{ij}\hbox{\tiny or }} \atop \phi_{k\ell}<\phi_{kj}} \bar y_\ell)$. If $\bar \rho_{ik} > \bar \theta_{ik}$, then, incorporate the cut:
$$\theta_{ik}\ge \sum_{j\in J_{ik}} |\phi_{ij}-\phi_{kj}|(y_j - \sum_{{\ell \in M :\atop \phi_{i\ell}<\phi_{ij}\hbox{\tiny or }} \atop \phi_{k\ell}<\phi_{kj}} y_\ell).$$

\subsection{$k$-sums based formulation}

With the same ideas applied to reformulate the continuous problem  into an ordered median problem,
we get that:
$$
\dsum_{i\in N}\dsum_{k\in N} \IE_{ik}(\mathbf{X})
 = 2 \dsum_{j\in P} \dsum_{k\in N}  \dsum_{\ell=1}^k \phi_{(\ell)j}  - \dsum_{j\in P}  (2n+1-k_j) \dsum_{\ell=1}^{n} \phi_{\ell j} x_{\ell j}
$$
where $k_j$ is the number of customers allocated to plant $j$. Thus the following formulation that we call (${\rm M}_3^D$) is valid for the problem:

\begin{align}
\min &\;\; 2 \dsum_{j\in P} \dsum_{k\in N} \left(\dsum_{\ell \in N} u_{klj} + \dsum_{i\in N} v_{kij}\right) - (2n+1)\dsum_{j\in P}\dsum_{i \in N} \phi_{ij} x_{ij} + \dsum_{j\in P}\dsum_{i \in N} \phi_{ij}\alpha_{ij} \label{m3:obj}\\
\mbox{s.t. } &  \sum_{j\in M} y_j = p, \ \ \label{pmd:1}\\
& \sum_{j\in M} x_{ij} =1, \ \ \forall i\in N, \label{pmd:2}\\
& x_{ij}\le y_j, \ \ \forall i\in N,\ j\in M, \label{pmd:3}\\
&y_j + \sum_{\ell \in M:\atop \phi_{ij}<\phi_{i\ell}} x_{i\ell} \le 1, \ \ \forall i\in N,\ j\in M, \label{cacd}\\
& \alpha_{ij} \geq \dsum_{\ell\in N} x_{\ell j} -  (n-p)(1-x_{ij}),\ \ \forall i, \ell\in N, j\in P,\label{ompd:14}\\
& u_{k\ell j} + v_{kij} \geq \phi_{ij} x_{ij},\ \ \forall i, \ell, k\in N (\ell\leq k), j\in P,\label{ompd:10}\\
& u_{kij}, v_{kij}\geq 0,\ \ \forall k, i \in N, j\in P,\\
& x_{ij} \in \{0,1\},\ \ \forall i\in N, j\in P,\\
& \alpha_{ij} \geq 0,\ \ \forall i, \ell\in N, j\in P.\label{m3:f}
\end{align}

\section{Computational study} \label{resultados}

In this section we provide the results of our computational experience in order to evaluate the performance of the proposed approaches. All the formulations were coded in Python 3.7 in an iMac with 3.3GHz with an Intel Core i7 with 4 cores and 16GB 1867 MHz DDR3 RAM. We used Gurobi 9.1.2 as optimization solver. A time limit of 2 hours was fixed for all the instances.

In order to produce a set of test instances, we use a similar strategy than in \cite{Espejo&Marin&Puerto&RodriguezChia:2010}. We generate two different types of instances for each combination of parameters, $n$, $p$ and $d$. For the instances of type \texttt{random}, the demand points are uniformly generated in  $[0,100]^d$. In the instances of type \texttt{blob} the points where generated as isotropic Gaussian blobs in $[0,100]^d$ with $\lceil \frac{n}{3} \rceil$ cluster centers and standard deviation $1$. Whereas uniform  random instances are the most popular when generating random instances to test Location Science problems, the blob instances more adequately represent the behavior of users which are usually geographically clustered, simulating a higher concentration of users around certain points of interests. The generated instances are available at \url{https://github.com/vblancoOR/intraenvy}.

For the discrete problem, the set of potential sites for the facilities are assumed to be the whole set of demand points and the cost matrix $\Phi$ is pre-computed using the $\ell_1$-norm.

We tested the formulations on a testbed of five instances for each combination of type (\texttt{random} and \texttt{blob}), $d\in \{2,3\}$, $n\in \{10, 20, 30, 40, 50\}$ and different values of $p \in \{2,3,5,7,10, 15, 20, 25, 30, 35, 40\}$ with $p \leq \frac{3n}{4}$. In total, $1400$ different instances where generated.

For each of the instances, we run the following formulations for the minimum $p$-Intra Envy Facility Location problem.

\begin{center}
\begin{tabular}{cc|cc}
\multicolumn{2}{c|}{Discrete} & \multicolumn{2}{c}{Continuous}\\\hline
\texttt{M1}$^D$: & \eqref{pm:0}-\eqref{pm:f} &\texttt{M1}$^C$: &\eqref{cmiemp:obj}-\eqref{cmiemp:f}\\
\texttt{M2}$^D$:  &\eqref{m2:obj}-\eqref{m2:f} & \texttt{M2}$^C$: &\eqref{omp:obj}-\eqref{omp:f}\\
\texttt{M3}$^D$: & \eqref{m3:obj}-\eqref{m3:f} & \texttt{M3}$^C$: &\eqref{ot:obj}-\eqref{ot:f}\\\hline
\end{tabular}
\end{center}

\subsection{Median, Envy and IntraEnvy Measures}

The first experiment that we run is devoted to determine the convenience of the \emph{intra-envy} model in our instances. For each instance, we solve three different Facility Location models for each of the two different solution domains (discrete and continuous) that differ in their optimization criterion, namely, $p$-median, envy and intra-envy. In the $p$-median model, the goal is to minimize the overall sum of the distances from each user to its closest facility, i.e.:
\begin{equation}\label{median}\tag{Median}
\min_{\mathbf{X} \subseteq \X:\atop |\mathbf{X}|=p} \dsum_{i\in N} \min_{j=1, \ldots, p} \Phi_i(X_{j(i)}).
\end{equation}
The (global) envy model aims to minimize the overall envy felt by the users, no matter to which facility is allocated, that is, defining the envy between two users $i, k \in N$ for a given set of facilities $\mathbf{X} \subset \X$ as:
$$
{\rm Envy}_{ik}(\mathbf{X})  = \begin{cases}
 	\Phi_i(X_{j(i)})-\Phi_k(X_{j(k)}) & \mbox{if $\Phi_k(X_{j(i)})<\Phi_i(X_{j(k)})$,}\\
   0 & \mbox{otherwise.}  
\end{cases}
$$
the envy problem consists of:
\begin{equation}\label{median}\tag{Envy}
\min_{\mathbf{X} \subseteq \X:\atop |\mathbf{X}|=p} \dsum_{i\in N} \dsum_{k\in N}  {\rm Envy}_{ik}(\mathbf{X}).
\end{equation}

For each of these solutions, we evaluate the three different objective functions and analyze the results which are shown in figures \ref{fig:Dev_IE_D} to \ref{fig:Dev_E_C}. In these figures, the results for the $p$-median problem are highlighted in green color, those for the envy problem in orange and the results of the intraenvy problem in blue color.

In figures \ref{fig:Dev_IE_D} and \ref{fig:Dev_IE_C} we report the average deviations (by the values of $n$ and $p$), in terms of intra-envy, on the solutions of the models that do not consider such a criterion in their objectives, namely $p$-median/multi-facility Weber and envy, for both the discrete and the continuous instances. The first observation that can be drawn is that the solutions of the Intra-Envy model do not coincide with those obtained with the other models. Specifically, for the discrete instances one can find instances for which the $p$-median problem obtains solutions deviated from the optimal intra-envy in more than $20\%$. For the continuous instances the difference is even more impresive with deviations close to $100\%$ (caused by intra-envy values close to zero). For the discrete instances this deviation decreases with the values of $p$, whereas for the continuous instances the situation is opposite, the larger the $p$ the larger the deviation. It is caused by the fact that the continuous instances allow more flexibility and an overall intra-envy close to zero as $p$ increases. Observe that a zero intra-envy solution is one in which all the users allocated to a facility assume the same allocation costs. If the closest-assignment constraints were not present, it could be obtained by \textit{clustering} users in the same $\ell_1$-norm orbit around a point (the center). Since we also force the users to be allocated to their closest facility, a close to zero-envy solution can be also obtained by co-locating all the facilities at the same position and deciding the clusters of users by similar distances to the single center. This situation is not possible in the discrete instances, but in the continuous instances, as far as $p$ increases, the facilities tend to co-locate. Nevertheless, this is not the situation of the median objective, where co-location is not a valid strrategy.

The high deviation of the median problem with respect to the intraenvy measure can be also observed when averaging by the value of $n$. There, whereas the envy model seems to obtain stable solutions with respect to the intra-envy, the median problem is in average deviated from the intraenvy in more than 10\% for the discrete instances and more than 70\% for the continuous instances. We observed that the deviation for the random instances is $10\%$ smaller than the obtained for the blobs instances.

Figures \ref{fig:Dev_M_D} and \ref{fig:Dev_M_C}  show the results of measuring the median objective (overall allocation costs) on the envy-based solutions with respect to the $p$-median solutions for the two types of instances. This deviation is known as the price of fairness according to Bertsimas et. al~\cite{Bertsimas&Farias&Trichakis:2011}. As expected, solutions with small global envy result in solutions with higher overall median-based allocation costs. Nevertheless, the deviations in these costs of the obtained solutions for the discrete instances are not that large, being the overall extra cost of the intra-envy model less than $9\%$ with respect to the best $p$-median solution. For values of $p\geq 10$, this extra cost reduces, in average, to $4\%$. The performance of the continuous instances is again different with respect to the efficiency. On the one hand, the average deviations are larger than 35\% in all cases, being more than 90\% for the largest values of $p$. Nevertheless, the behaviour of the envy and the intra-envy solutions is similar, being the envy models slightly more efficient (in average $10\%$ for the continuos instances and $2\%$ for the discrete ones) than the intra-envy model.

Finally, in figures \ref{fig:Dev_E_D} and \ref{fig:Dev_E_C} we evaluate the $p$-median and the intra-envy model in terms of the global envy objective. In both types of instances, although the $p$-median model seems to deviate more from the envy model for small values of $p$, the deviations for larger $p$ are neglectable for the discrete instances, being the overall global envy for this model similar to those of the envy model. Nevertheless, the intra-envy model results in solutions that differ from those of the envy model, being the deviation consistently closed to its average (around $3.5\%$). 

Summarizing, the solutions obtained with the intra-envy model clearly differ from those obtained with the classical $p$-median and the envy model proposed in \cite{Espejo&Marin&Puerto&RodriguezChia:2010}. Although determining facilities that exhibit minimum intra-envy has a direct impact in the transportation costs of the solution, in the discrete instances these extra costs are small enough to sacrifice them in case one desire to avoid envies among the customers allocated to a same facility, whereas in the continuous instances they are similar to those obtained with the envy model.

\begin{figure}[h]
\includegraphics[scale=0.8]{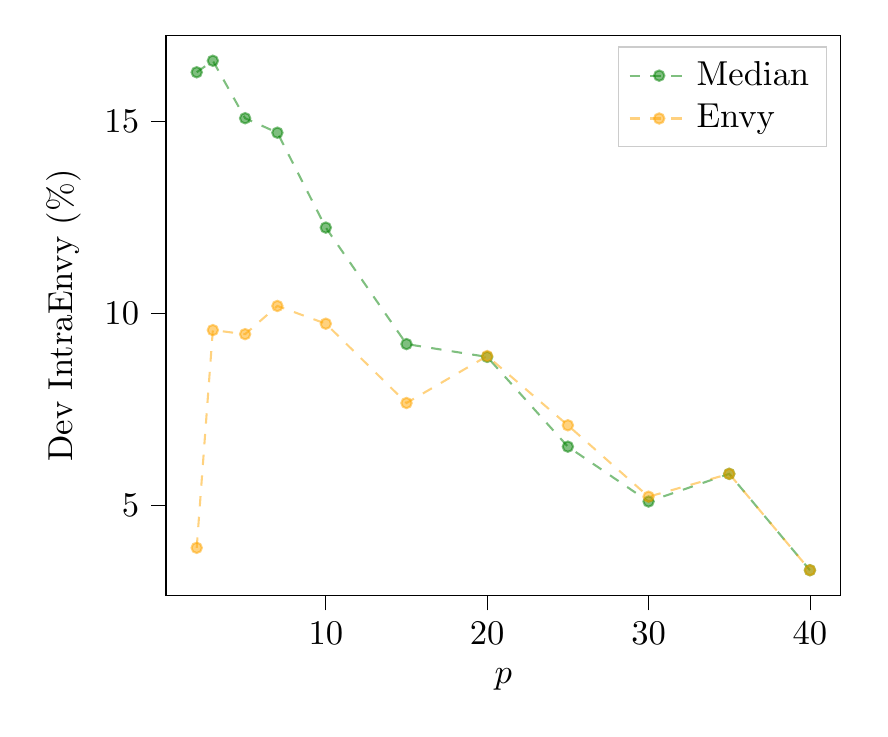}~\includegraphics[scale=0.8]{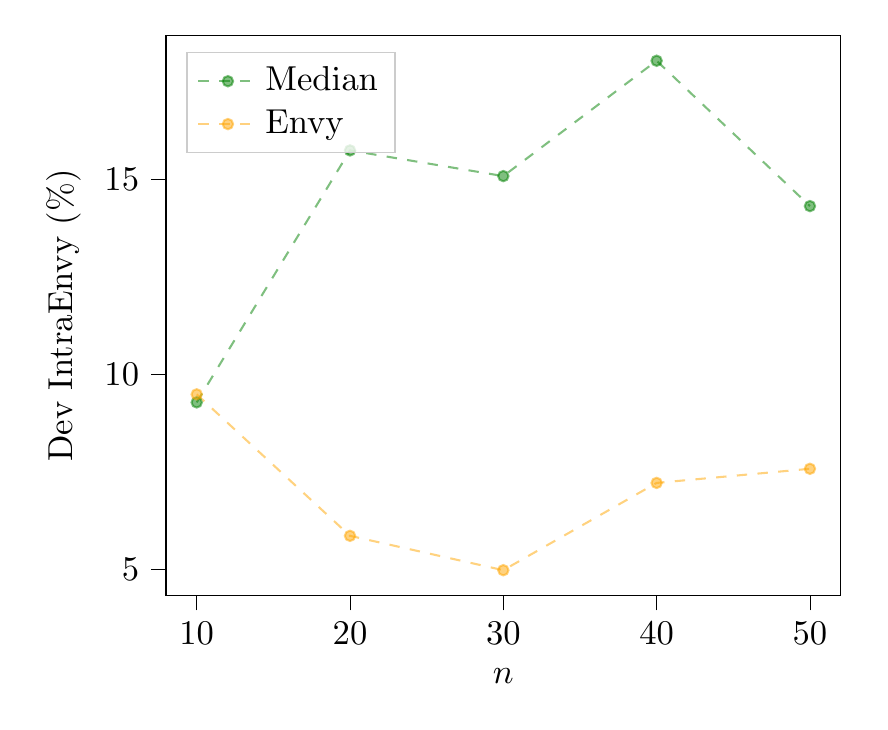}
\caption{Average deviations of $p$-median and Envy solutions in the intra-envy measure with respect to the best intra-envy solution by $p$ (left) and $n$ (right) parameters for the discrete instances.\label{fig:Dev_IE_D}}
\end{figure}

\begin{figure}[h]
\includegraphics[scale=0.8]{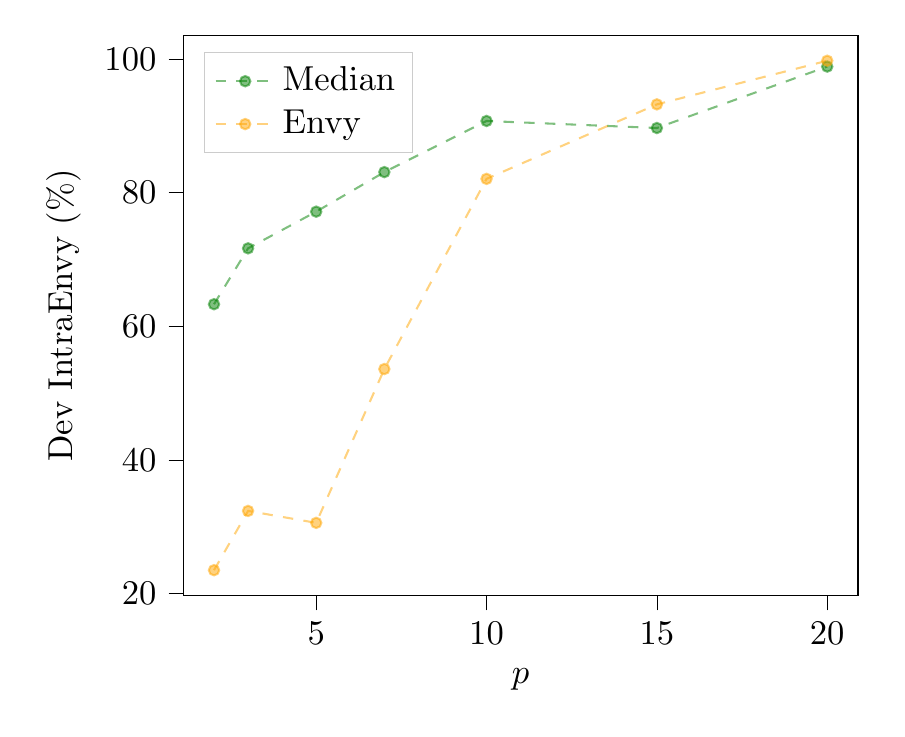}~\includegraphics[scale=0.8]{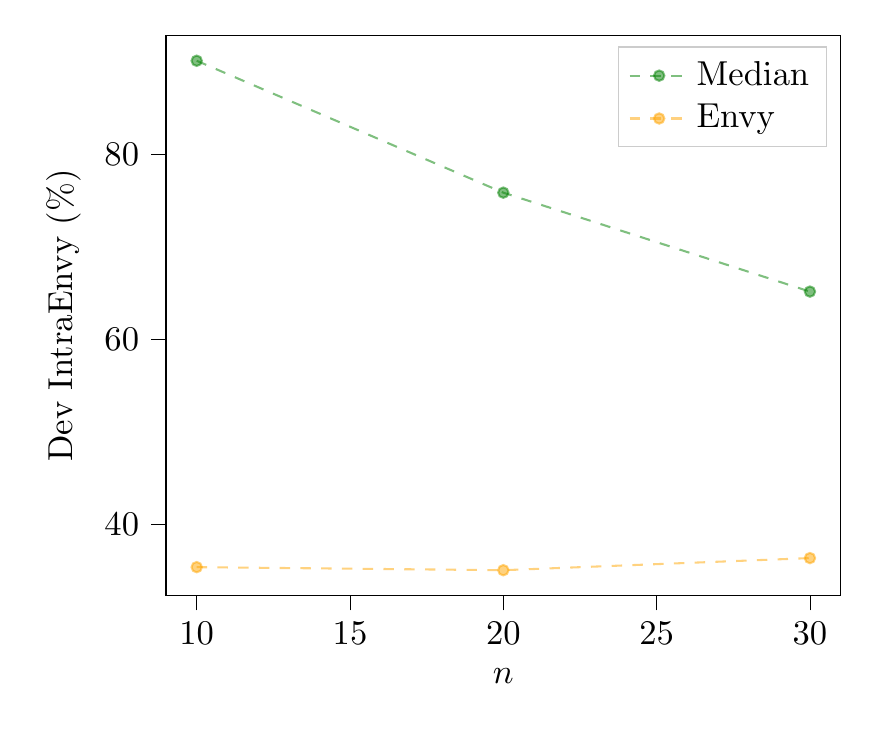}
\caption{Average deviations of $p$-median and Envy solutions in the intra-envy measure with respect to the best intra-envy solution by $p$ (left) and $n$ (right) parameters  for the continuous instances.\label{fig:Dev_IE_C}}
\end{figure}

\begin{figure}[h]
\includegraphics[scale=0.8]{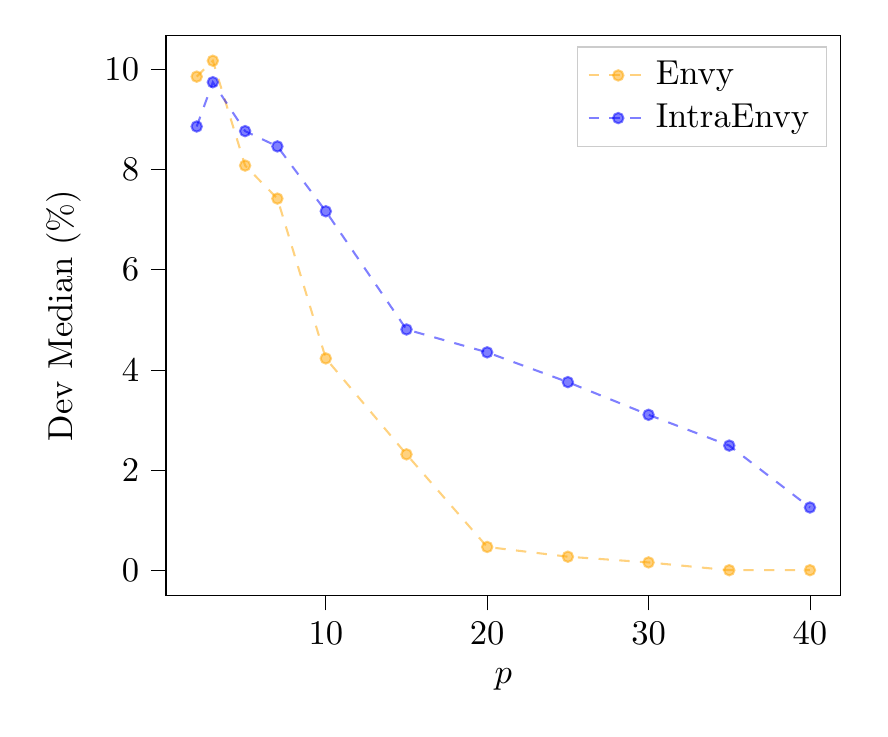}~\includegraphics[scale=0.8]{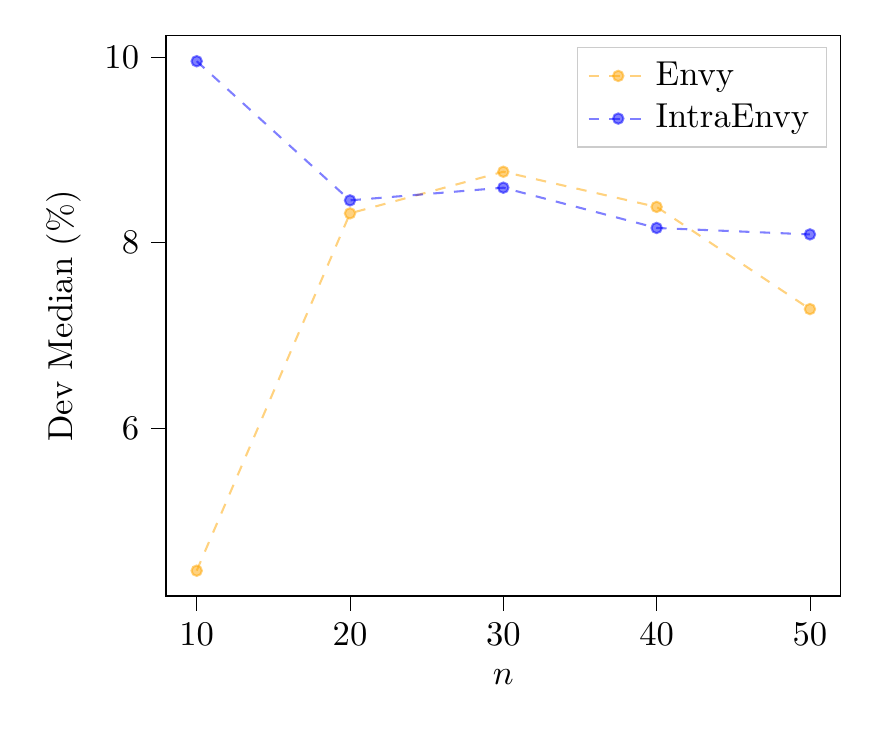}
\caption{Average deviations of Envy and Intra-Envy solutions in the median measure with respect to the best median solution by $p$ (left) and $n$ (right) parameters for the discrete instances.\label{fig:Dev_M_D}}
\end{figure}

\begin{figure}[h]
\includegraphics[scale=0.8]{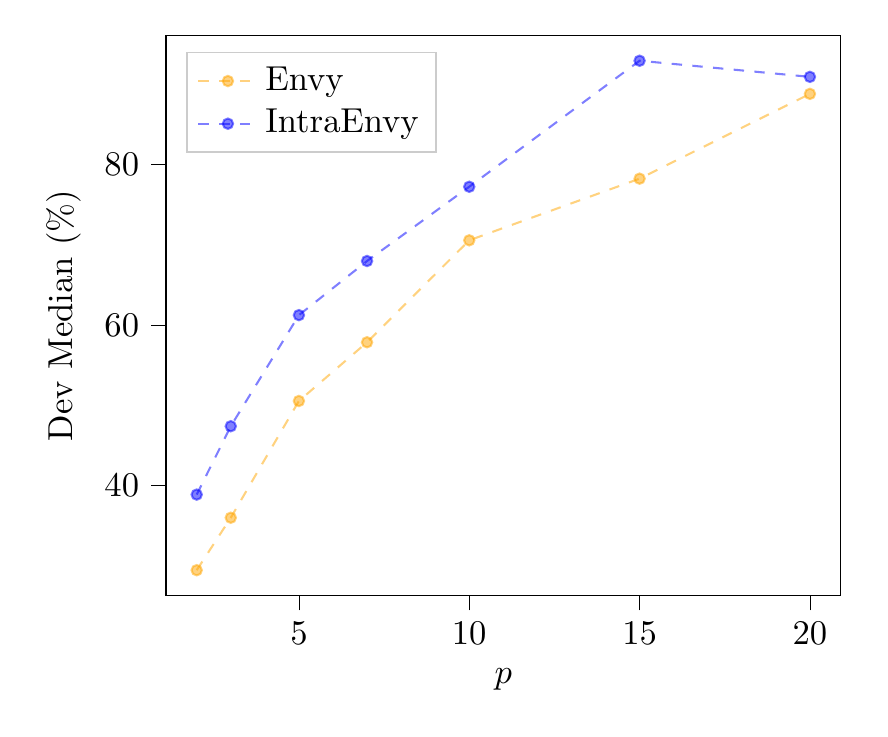}~\includegraphics[scale=0.8]{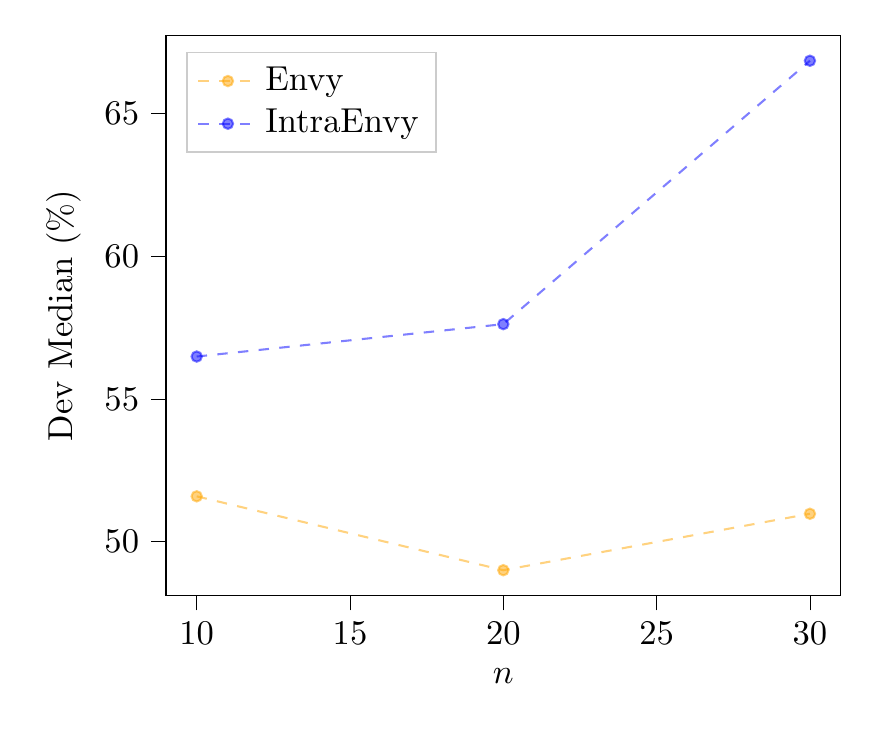}
\caption{Average deviations of Envy and Intra-Envy solutions in the median measure with respect to the best median solution by $p$ (left) and $n$ (right) parameters for the continuous instances.\label{fig:Dev_M_C}}
\end{figure}

\begin{figure}[h]
\includegraphics[scale=0.8]{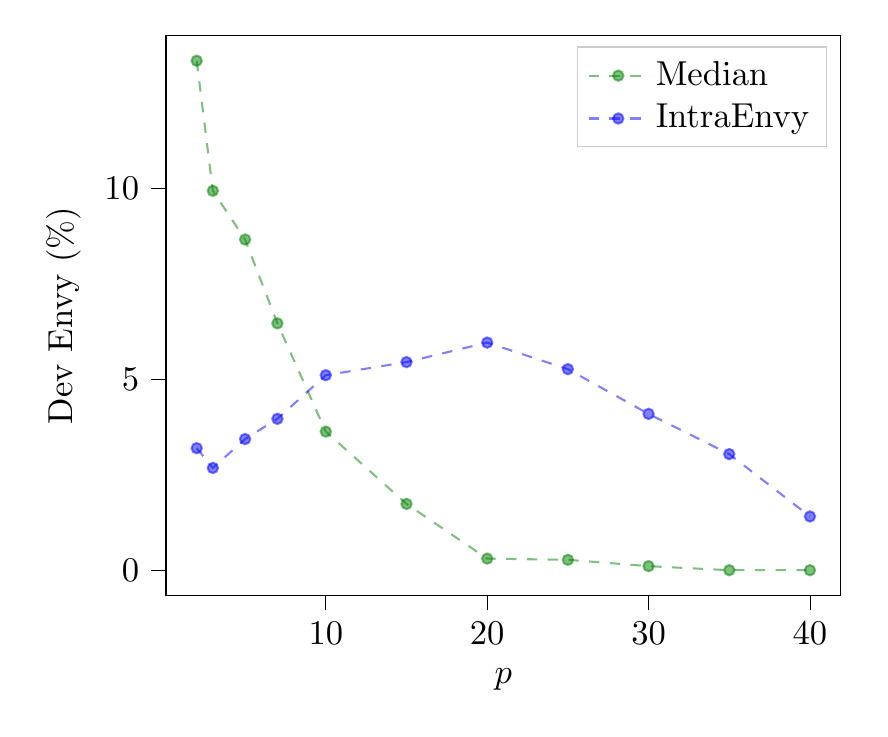}~\includegraphics[scale=0.8]{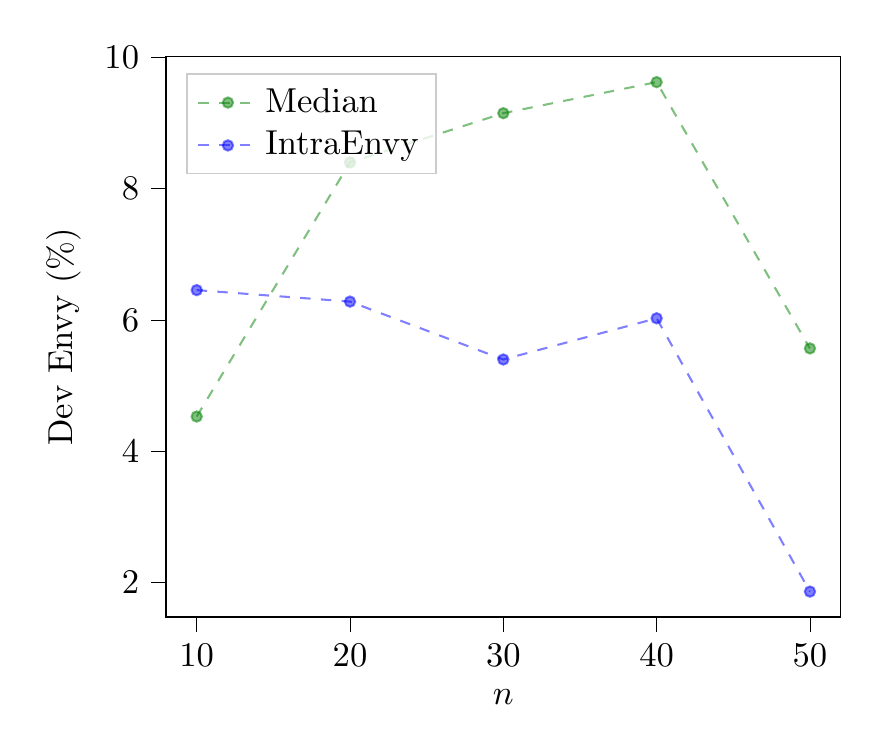}
\caption{Average deviations of $p$-median and Intra-Envy solutions in the global-envy measure with respect to the best envy solution by $p$ (left) and $n$ (right) parameters for the discrete instances.\label{fig:Dev_E_D}}
\end{figure}

\begin{figure}[h]
\includegraphics[scale=0.8]{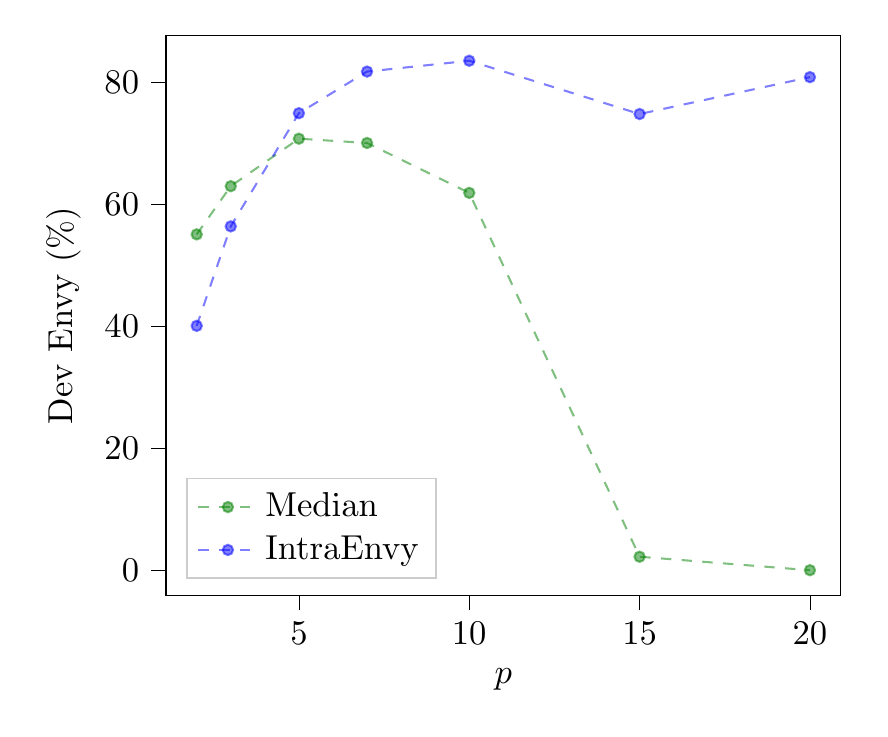}~\includegraphics[scale=0.8]{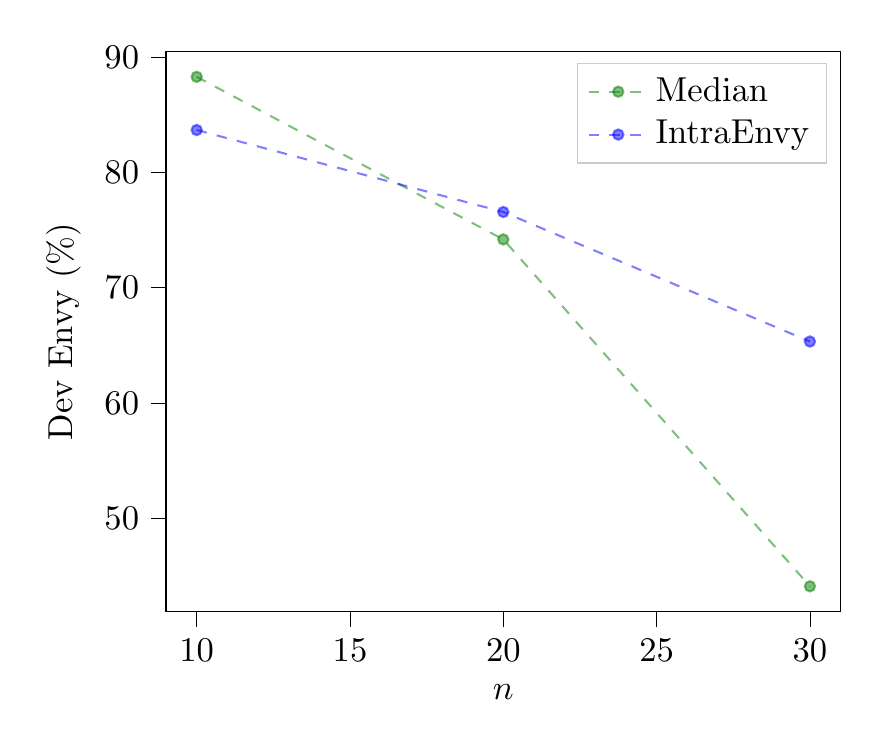}
\caption{Average deviations of $p$-median and Intra-Envy solutions in the global-envy measure with respect to the best envy solution by $p$ (left) and $n$ (right) parameters for the continuous instances.\label{fig:Dev_E_C}}
\end{figure}

\subsection{Computational Performance}

In this section we analyze the computational performance of the different intra-envy formulations that we propose here. Specifically, we study the computational difficulty of solving each of the formulations by means of the consumed CPU time and the MIPGap, for those instances that were not able to be optimally solved within the time limit. We summarize here the obtained results, whereas the detailed results are available for the interested reader in \url{https://github.com/vblancoOR/intraenvy}.

In Table \ref{t:disc1} we show the average CPU time, in seconds, required by the different formulations for solving the discrete location instances with $n$ up to $30$. We average our results by $n$ (number of users), $p$ (number of facilities to be located), type of instance (\texttt{random} or \texttt{blobs}) and integer programming formulation for the $p$-intra envy problem (\texttt{M1}, \texttt{M2}, and \texttt{M3}). In case the none of the averaged instances where solved to proven optimality within the time limit we flag them with \texttt{TL}. The results of solving 300 different instances are summarized.

\begin{table}[htbp]
  \centering
    \begin{tabular}{cr|rrr|rrl|}    &       & \multicolumn{3}{c|}{\texttt{BLB}} & \multicolumn{3}{c|}{\texttt{RND}} \\
   $n$& $p$& \texttt{M1}$^D$ &\texttt{M2}$^D$ & \texttt{M3}$^D$ &	\texttt{M1}& \texttt{M2}$^D$ & \texttt{M3}$^D$ \\
    \hline
    \multirow{3}[2]{*}{10} & 2     & 0.01  & 0.03  &   0.56    & 0.02  & 0.03  &{0.69} \\
          & 3     & 0.01  & 0.03  &   0.61    & 0.03  & 0.03  & 0.63 \\
          & 5     & 0.01  & 0.03  &    0.61   & 0.01  & 0.03  & 0.70 \\
    \hline
    \multirow{5}[2]{*}{20} & 2     & 1.89  & 0.32  &    307.16    & 1.82  & 0.40  &{737.49} \\
          & 3     & 1.05  & 0.42  &   76.28    & 1.42  & 0.53  & {807.96} \\
          & 5     & 0.14  & 0.51  &  1050.24     & 0.54  & 1.00  &{6585.21} \\
          & 7     & 0.06  & 1.01  &   4247.95    & 0.19  & 1.22  &{6525.11} \\
          & 10    & 0.05  & 1.14  &   5457.84    & 0.08  & 4.34  &{6427.51} \\
    \hline
    \multirow{7}[2]{*}{30} & 2     & 402.80 & 2.90  &    \texttt{TL}   & 412.41 & 103.39 & \texttt{TL} \\
          & 3     & 392.30 & 5.36  &   \texttt{TL}    & 131.35 & 9.74  & \texttt{TL} \\
          & 5     & 4.38  & 125.07 &   \texttt{TL}    & 9.44  & 635.34 & \texttt{TL} \\
          & 7     & 1.34  & 164.79 &  \texttt{TL}     & 4.34  & 1173.37 & \texttt{TL} \\
          & 10    & 0.32  & 222.33 & \texttt{TL}      & 1.28  & 2504.56 & \texttt{TL} \\
          & 15    & 0.23  & 364.95 &  \texttt{TL}     & 0.39  & 1231.61 & \texttt{TL} \\
          & 20    & 0.18  & 32.88 &   \texttt{TL}    & 0.23  & 378.70 & \texttt{TL} \\
    \hline
    \end{tabular}%
  \caption{CPU Times for discrete instances with $n\in \{10,20,30\}$.}
  \label{t:disc1}%
\end{table}%

One can observe from these results that model \texttt{M3}, the one based on the ordered median reformulation of the problem, is not competitive with the other formulations. Specifically, even for instances with $n=20$, formulation \texttt{M2} required much more CPU time to solve the instances than the rest of the formulations, for both the random and the blobs instances. Note that, the ordered median representation of the problem requires $3n^2p$ auxiliary variables ($u$- and $v$-variables in \eqref{m3:obj} -\eqref{m3:f}) apart from the $np + n$ decision variables ($x$ and $y$) whereas \texttt{M1} (resp. \texttt{M2}) uses $n^2$ auxiliary variables ($\theta$) and the $np+p$ (resp. $n$) decision variables. We do not observe significative different in the behaviour of the three formulations neither in the two dimensions ($d=2, 3$) nor the type of instance (\texttt{BLB} or \texttt{RND}). 

All the instances of these sizes were solved up to optimality with formulations \texttt{M1} and \texttt{M2}. Nevertheless, none of the instances with $n=30$ were optimally solved with \texttt{M3} and only $13\%$ of the instances with $n=20$. 

Comparing \texttt{M1} and \texttt{M2} for these instances, it seems that \texttt{M2} requires less CPU time for solving the instances with small values of $p$ ($p\in \{2,3\}$) whereas \texttt{M1} consumes less CPU time than \texttt{M2} for instances with larger values of $p$ ($p\geq 5$). 

In Table \ref{t:disc2}  we show the results obtained with formulations \texttt{M1} and \texttt{M2} for the larger instances. We report there, for both formulations, the consumed CPU time in seconds, the the percent of instances not optimally solved within the time limit, the percent MIPGap and the percent deviation of the best obtained solutions with respect to the relaxed solution after exploring the root node of the branch and bound tree. 

The same behavior of \texttt{M1} and \texttt{M2} was observed for these instances, that is, \texttt{M2} outperforms \texttt{M1} for small values of $p$ whereas \texttt{M1} obtained better results for larger values of $p$. This perforrmance is also observed in the number of unsolved instances and the MIPGap. 

The information about the rootgap provides details about the weakness of the relaxed polyhedron induced by \texttt{M2} with respect to \texttt{M1} which makes more difficult to solve larger instances.

\begin{table}[htbp]
  \centering
    \begin{tabular}{cr|rr|rr|rr|rr|}    &       & \multicolumn{2}{c|}{\textbf{CPU Time}} & \multicolumn{2}{c|}{\textbf{\%Unsolved}} &\multicolumn{2}{c|}{\textbf{\%MIPGap}}  & \multicolumn{2}{c|}{\textbf{\%RootGap}}  \\
   $n$& $p$& \texttt{M1}$^D$ &\texttt{M2}$^D$  &	\texttt{M1}$^D$& \texttt{M2}$^D$ &\texttt{M1}$^D$ &\texttt{M2}$^D$&\texttt{M1}$^D$ &\texttt{M2}$^D$ \\\hline
      \cline{1-6}\cline{9-10}    \multirow{9}[2]{*}{40} & 2     & 1558.92 & 182.05 & 0\%   & 0\%   & 0\%   & 0\%   & 54\%  & 100\%\\
          & 3     & 2024.13 & 380.26 & 0\%   & 0\%   & 0\%   & 0\%   & 46\%  & 100\% \\
          & 5     & 1008.44 & 3944.59 & 0\%   & 25\%  & 0\%   & 10\%  & 30\%  & 100\% \\
          & 7     & 421.95 & 6235.06 & 0\%   & 60\%  & 0\%   & 33\%  & 21\%  & 100\% \\
          & 10    & 196.59 & 7088.01 & 0\%   & 95\%  & 0\%   & 54\%  & 9\%   & 100\% \\
          & 15    & 2.00  & \texttt{TL} & 0\%   & 100\% & 0\%   & 57\%  & 1\%   & 100\% \\
          & 20    & 0.97  & \texttt{TL} & 0\%   & 100\% & 0\%   & 53\%  & 1\%   & 100\% \\
          & 25    & 0.66  & 6977.39 & 0\%   & 85\%  & 0\%   & 29\%  & 0\%   & 100\% \\
          & 30    & 0.53  & 2488.75 & 0\%   & 10\%  & 0\%   & 3\%   & 0\%   & 100\%\\\hline
\multirow{11}[2]{*}{50} & 2     & 6756.18 & 558.82 & 80\%  & 0\%   & 49\%  & 0\%   & 65\%  & 100\% \\
          & 3     & 7090.85 & 2933.55 & 95\%  & 0\%   & 48\%  & 0\%   & 57\%  & 100\% \\
          & 5     & 6619.14 & 6994.48 & 70\%  & 90\%  & 21\%  & 70\%  & 40\%  & 100\% \\
          & 7     & 4435.05 & \texttt{TL} & 40\%  & 100\% & 6\%   & 89\%  & 30\%  & 100\% \\
          & 10    & 1881.96 & \texttt{TL} & 0\%   & 100\% & 0\%   & 90\%  & 19\%  & 100\% \\
          & 15    & 260.42 & \texttt{TL} & 0\%   & 100\% & 0\%   & 90\%  & 4\%   & 100\% \\
          & 20    & 4.25  & \texttt{TL} & 0\%   & 100\% & 0\%   & 86\%  & 1\%   & 100\% \\
          & 25    & 2.76  & \texttt{TL} & 0\%   & 100\% & 0\%   & 78\%  & 0\%   & 100\% \\
          & 30    & 1.88  & \texttt{TL} & 0\%   & 100\% & 0\%   & 71\%  & 0\%   & 100\% \\
          & 35    & 1.40  & \texttt{TL} & 0\%   & 100\% & 0\%   & 59\%  & 1\%   & 100\% \\
          & 40    & 1.14  & 5266.67 & 0\%   & 45\%  & 0\%   & 17\%  & 1\%   & 100\%\\
    \hline
    \end{tabular}%
    
  \caption{Computational results for discrete instances with $n\in \{40, 50\}$.}
  \label{t:disc2}%
\end{table}%

In Table \ref{t:cont}  we report the results obtained for the continuous instances, organized similarly to those of the largest discrete instances for the three formulations we propose. We observe that the first formulation, \texttt{M1}, seems to have a better performance than the others in both consumed CPU time and number of optimally solved instances. Concretely, \texttt{M1} was able to solve $47\%$ of the instances whereas \texttt{M2} and \texttt{M3} only solved $26\%$ and $27\%$ of them, respectively. Moreover, the MIPGap for the unsolved instances within the time limit was smaller in the case of  \texttt{M1} .  As expected, the MIPGap and the root relaxation gaps were very large since the nonconvex terms that appear in the formulation to represent the closest distance between users and facilities were reformulated using {\em big M} constraints. Regarding CPU times, the \textit{easiest} instances seem to be those with large values of $p$, even if they use the largest number of variables for a fixed value of $n$. These times were consistently paired with smaller root gap relaxations.

Comparing the discrete and the continuous instances, the latter exhibit a higher computational difficulty: Only a few instances with 30 users could be solved up to optimality, whereas in the discrete case this size extended to 50 users.

\begin{table}[htbp]
  \centering
  {\small  \begin{tabular}{cr|rrr|rrr|rrr|rrr|}    &       & \multicolumn{3}{c|}{\textbf{CPU Time}} & \multicolumn{3}{c|}{\textbf{\%Unsolved}} &\multicolumn{3}{c|}{\textbf{\%MIPGap}}  & \multicolumn{3}{c|}{\textbf{\%RootGap}}  \\
   $n$& $p$& \texttt{M1}$^C$ &\texttt{M2}$^C$ &\texttt{M3}$^C$ &	\texttt{M1}$^C$& \texttt{M2}$^C$ & \texttt{M3}$^C$ &\texttt{M1}$^C$ &\texttt{M2}$^C$&\texttt{M3}$^C$ & \texttt{M1}$^C$ &\texttt{M2}$^C$ &\texttt{M3}$^C$ \\\hline
  \multirow{3}[2]{*}{10} & 2     & 20.81 & 68.76 & 108.92 & 0\%   & 0\%   & 0\%   & 0\%   & 0\%   & 0\%   & 100\% & 100\% & 100\%\\
          & 3     & 22.72 & 286.76 & 398.21 & 0\%   & 0\%   & 0\%   & 0\%   & 0\%   & 0\%   & 50\%  & 50\%  & 50\% \\
          & 5     & 5.41  & 9.11  & 17.60 & 0\%   & 0\%   & 0\%   & 0\%   & 0\%   & 0\%   & 0\%   & 0\%   & 0\%\\\hline
 \multirow{5}[2]{*}{20} & 2     & 2017.43 & \texttt{TL} & \texttt{TL} & 15\%  & 100\% & 100\% & 8\%   & 100\% & 100\% & 100\% & 100\% & 100\%\\
          & 3     & 6681.26 & \texttt{TL} & \texttt{TL} & 85\%  & 100\% & 100\% & 77\%  & 100\% & 100\% & 100\% & 100\% & 100\% \\
          & 5     & \texttt{TL} & \texttt{TL} & \texttt{TL} & 100\% & 100\% & 100\% & 100\% & 100\% & 100\% & 100\% & 100\% & 100\% \\
          & 7     & 5085.74 & 6557.49 & 6359.81 & 60\%  & 80\%  & 85\%  & 60\%  & 80\%  & 84\%  & 60\%  & 80\%  & 80\% \\
          & 10    & 1104.34 & 3172.46 & 4008.00 & 5\%   & 25\%  & 45\%  & 5\%   & 21\%  & 45\%  & 5\%   & 20\%  & 45\%\\\hline
\multirow{7}[2]{*}{30}  & 2     & 7151.42 & \texttt{TL} & \texttt{TL} & 95\%  & 100\% & 100\% & 75\%  & 100\% & 100\% & 95\%  & 90\%  & 95\% \\
          & 3     & \texttt{TL} & \texttt{TL} & \texttt{TL} & 100\% & 100\% & 100\% & 100\% & 100\% & 100\% & 100\% & 100\% & 95\% \\
          & 5     & \texttt{TL} & \texttt{TL} & \texttt{TL} & 100\% & 100\% & 100\% & 100\% & 100\% & 100\% & 100\% & 100\% & 100\% \\
          & 7     & \texttt{TL} & \texttt{TL} & \texttt{TL} & 100\% & 100\% & 100\% & 100\% & 100\% & 100\% & 100\% & 70\%  & 95\% \\
          & 10    & 6630.26 & \texttt{TL} & \texttt{TL} & 90\%  & 100\% & 100\% & 90\%  & 100\% & 100\% & 90\%  & 95\%  & 95\% \\
          & 15    & 5453.78 & \texttt{TL} & 6689.83 & 65\%  & 100\% & 85\%  & 65\%  & 100\% & 85\%  & 60\%  & 50\%  & 85\% \\
          & 20    & 3615.18 & \texttt{TL} & 6543.22 & 35\%  & 100\% & 85\%  & 35\%  & 100\% & 85\%  & 25\%  & 55\%  & 85\%\\\hline    \end{tabular}}%
  \caption{Computational results for continuous instances.}
  \label{t:cont}%
\end{table}%

\section{Conclusions}

We introduce in this paper the intra-Envy $p$-facility location problem in order to  determine the optimal position of $p$ services by minimizing the envy felt by the users allocated to the same facility. This problem allows to find local fair solutions of $p$-facility location taking into account the realistic assumption that users are not usually compared with all rest of users but with those that make use of the same facility. Furthermore we provide a general framework for the problem which is valid for the two most popular solution domains in facility location, discrete and continuous.

We derive different MILP formulations for the discrete and continuous versions of the problem, assuming that the distance measure for the continuous problems is the $\ell_1$-norm. The results of an extensive battery of computational experiments are reported. Apart from comparing computationally the different formulations, we analyze the solutions evaluating the proposed intra-envy measure, the global envy and the median functions.

Future research on the topic includes the study of valid inequalities for the different models that we propose. For larger instances, it would be helpful to design heuristic approaches that assure good quality solution in smaller computing times.

\section*{Acknowledgements}

Alfredo Mar\'in has been supported by Spanish Ministry of Science and Innovation under project PID2019-110886RB-I00. Part of this research was conducted while he was on sabbatical at IMUS in Universidad de Sevilla, Spain. V\'ictor Blanco and Justo Puerto have been partially supported by the Agencia Estatal de Investigación (AEI) and the European Regional Development’s fund (ERDF): PID2020-114594GB-C2;  Regional Government of Andalusia P18-FR-1422 and B-FQM-322-UGR20 (ERDF),  and Fundación BBVA: project NetmeetData (Ayudas Fundación BBVA a equipos de investigación científica 2019). V\'ictor Blanco was also partially supported by the IMAG-Maria de Maeztu grant CEX2020-001105-M /AEI /10.13039/501100011033 and Ayudas de Recualificaci\'on UGR 2021 (NextGenerationEU).

\providecommand{\bysame}{\leavevmode\hbox to3em{\hrulefill}\thinspace}
\providecommand{\MR}{\relax\ifhmode\unskip\space\fi MR }
\providecommand{\MRhref}[2]{%
  \href{http://www.ams.org/mathscinet-getitem?mr=#1}{#2}
}
\providecommand{\href}[2]{#2}


\begin{thebibliography}{10}

\bibitem{Berman&Kaplan:1990}
Oded Berman and Edward~H Kaplan, \emph{Equity maximizing facility location
  schemes}, Transportation Science \textbf{24} (1990), no.~2, 137--144.

\bibitem{Bertsimas&Farias&Trichakis:2011}
Dimitris Bertsimas, Vivek~F Farias, and Nikolaos Trichakis, \emph{The price of
  fairness}, Operations Research \textbf{59} (2011), no.~1, 17--31.

\bibitem{Bertsimas&Farias&Trichakis:2012}
\bysame, \emph{On the efficiency-fairness trade-off}, Management Science
  \textbf{58} (2012), no.~12, 2234--2250.

\bibitem{Blanco&Gazquez:2022}
V\'ictor Blanco and Ricardo G{\'a}zquez, \emph{Fairness in maximal covering
  facility location problems}, Available at Arxiv, 2022.

\bibitem{Blanco&Puerto&ElHaj:2014}
V\'ictor Blanco, Justo Puerto, and Safae El-Haj Ben-Ali, \emph{Revisiting
  several problems and algorithms in continuous location with $\ell_p$ norms},
  Computational Optimization and Applications \textbf{58} (2014), no.~3,
  563--595.

\bibitem{Blanco&Puerto&ElHaj:2016}
\bysame, \emph{Continuous multifacility ordered median location problems},
  European Journal of Operational Research \textbf{250} (2016), no.~1, 56--64.

\bibitem{Brams&Fishburn:2000}
Steven~J Brams and Peter~C Fishburn, \emph{Fair division of indivisible items
  between two people with identical preferences: Envy-freeness,
  {P}areto-optimality, and equity}, Social Choice and Welfare \textbf{17}
  (2000), no.~2, 247--267.

\bibitem{Chanta&Mayorga&Kurz&McLay:2011}
Sunarin Chanta, Maria~E Mayorga, Mary~E Kurz, and Laura~A McLay, \emph{The
  minimum p-envy location problem: a new model for equitable distribution of
  emergency resources}, IIE Transactions on Healthcare Systems Engineering
  \textbf{1} (2011), no.~2, 101--115.

\bibitem{Chanta&Mayorga&McLay:2014}
Sunarin Chanta, Maria~E Mayorga, and Laura~A McLay, \emph{The minimum p-envy
  location problem with requirement on minimum survival rate}, Computers \&
  Industrial Engineering \textbf{74} (2014), 228--239.

\bibitem{Chun:2006}
Youngsub Chun, \emph{No-envy in queueing problems}, Economic Theory \textbf{29}
  (2006), no.~1, 151--162.

\bibitem{DallAglio&Hill:2003}
Marco Dall'Aglio and Theodore~P Hill, \emph{Maximin share and minimax envy in
  fair-division problems}, Journal of Mathematical Analysis and Applications
  \textbf{281} (2003), no.~1, 346--361.

\bibitem{dominguez2022mixed}
Concepci{\'o}n Dom{\'\i}nguez, Martine Labb{\'e}, and Alfredo Mar{\'\i}n,
  \emph{Mixed-integer formulations for the capacitated rank pricing problem
  with envy}, Computers \& Operations Research \textbf{140} (2022), 105664.

\bibitem{Eiselt:1995}
Horst~A Eiselt and Gilbert Laporte, \emph{Objectives in location problems},
  Facility location: a survey of applications and methods (Zvi Drezner and
  Horst~W. Hamacher, eds.), Springer, 1995, pp.~151--179.

\bibitem{Erkut:1993}
Erhan Erkut, \emph{Inequality measures for location problems.}, Computers \&
  Operations Research (1993).

\bibitem{Espejo&Marin&Puerto&RodriguezChia:2010}
Inmaculada Espejo, Alfredo Mar{\'\i}n, Justo Puerto, and Antonio~M
  Rodr{\'\i}guez-Ch{\'\i}a, \emph{A comparison of formulations and solution
  methods for the minimum-envy location problem}, Computers \& Operations
  Research \textbf{36} (2009), no.~6, 1966--1981.

\bibitem{Filippi&Guastttaroba&Speranza:2021}
Carlo Filippi, Gianfranco Guastaroba, and M~Grazia Speranza, \emph{On
  single-source capacitated facility location with cost and fairness
  objectives}, European Journal of Operational Research \textbf{289} (2021),
  no.~3, 959--974.

\bibitem{Garfinkel&Fernandez&Lowe:2006}
Robert Garfinkel, Elena Fern{\'a}ndez, and Timothy~J Lowe, \emph{The k-centrum
  shortest path problem}, Top \textbf{14} (2006), no.~2, 279--292.

\bibitem{Haake&Rait&Su:2002}
Claus-Jochen Haake, Matthias~G Raith, and Francis~Edward Su, \emph{Bidding for
  envy-freeness: A procedural approach to n-player fair-division problems},
  Social Choice and Welfare \textbf{19} (2002), no.~4, 723--749.

\bibitem{Hakimi:1964}
S~Louis Hakimi, \emph{Optimum locations of switching centers and the absolute
  centers and medians of a graph}, Operations Research \textbf{12} (1964),
  no.~3, 450--459.

\bibitem{Halpern&Maimon:1981}
Jonathan Halpern and Oded Maimon, \emph{Accord and conflict among several
  objectives in locational decisions on tree networks}, North-Holland,
  Amsterdam, 1981.

\bibitem{Laporte&Saldanha&Nickel:2019}
Gilbert Laporte, Francisco Saldanha-da Gama, , and Stephan Nickel,
  \emph{Location science}, 2019.

\bibitem{Marin&Nickel&Puerto:2009}
Alfredo Mar{\'\i}n, Stefan Nickel, Justo Puerto, and Sebastian Velten, \emph{A
  flexible model and efficient solution strategies for discrete location
  problems}, Discrete Applied Mathematics \textbf{157} (2009), no.~5,
  1128--1145.

\bibitem{Marin&Nickel&Velten:2010}
Alfredo Mar{\'\i}n, Stefan Nickel, and Sebastian Velten, \emph{An extended
  covering model for flexible discrete and equity location problems},
  Mathematical Methods of Operations Research \textbf{71} (2010), no.~1,
  125--163.

\bibitem{MarinPoncePuerto20}
Alfredo Mar{\'\i}n, Diego Ponce, and Justo Puerto, \emph{A fresh view on the
  discrete ordered median problem based on partial monotonicity}, European
  Journal of Operational Research \textbf{286} (2020), no.~3, 839--848.

\bibitem{Marsh&Schilling:1994}
Michael~T Marsh and David~A Schilling, \emph{Equity measurement in facility
  location analysis: A review and framework}, European Journal of Operational
  Research \textbf{74} (1994), no.~1, 1--17.

\bibitem{Meng&McCauley&Kaplan&Leung&Coskun:2015}
Jie Meng, Samuel McCauley, Fulya Kaplan, Vitus~J Leung, and Ayse~K Coskun,
  \emph{Simulation and optimization of hpc job allocation for jointly reducing
  communication and cooling costs}, Sustainable Computing: Informatics and
  Systems \textbf{6} (2015), 48--57.

\bibitem{Mesa&Puerto&Tamir:2003}
Juan~A Mesa, Justo Puerto, and Arie Tamir, \emph{Improved algorithms for
  several network location problems with equality measures}, Discrete Applied
  Mathematics \textbf{130} (2003), no.~3, 437--448.

\bibitem{Moulin:2014}
Herv{\'e} Moulin, \emph{Cooperative microeconomics: a game-theoretic
  introduction}, Princeton University Press, 2014.

\bibitem{Mulligan:1991}
Gordon~F Mulligan, \emph{Equality measures and facility location}, Papers in
  Regional Science \textbf{70} (1991), no.~4, 345--365.

\bibitem{Ogryczak&Tamir:2003}
Wlodzimierz Ogryczak and Arie Tamir, \emph{Minimizing the sum of the k largest
  functions in linear time}, Information Processing Letters \textbf{85} (2003),
  no.~3, 117--122.

\bibitem{Ohseto:2005}
Shinji Ohseto, \emph{Strategy-proof assignment with fair compensation},
  Mathematical Social Sciences \textbf{50} (2005), no.~2, 215--226.

\bibitem{Papai:2003}
Szilvia P{\'a}pai, \emph{Groves sealed bid auctions of heterogeneous objects
  with fair prices}, Social Choice and Welfare \textbf{20} (2003), no.~3,
  371--385.

\bibitem{PuertoChiaTamir17}
Justo Puerto, Antonio~M Rodr\'{\i}guez-Ch\'{\i}a, and Arie Tamir,
  \emph{Revisiting k-sum optimization}, Mathematical Programming \textbf{165}
  (2017), no.~2, 579--604.

\bibitem{Reijnierse&Potters:1998}
J~Hans Reijnierse and Jos~AM Potters, \emph{On finding an envy-free
  {P}areto-optimal division}, Mathematical Programming \textbf{83} (1998),
  no.~1, 291--311.

\bibitem{Savas:1978}
Emanuel~S Savas, \emph{On equity in providing public services}, Management
  Science \textbf{24} (1978), no.~8, 800--808.

\bibitem{Shioura&Sun&Yang:2006}
Akiyoshi Shioura, Ning Sun, and Zaifu Yang, \emph{Efficient strategy proof fair
  allocation algorithms}, Journal of the Operations Research Society of Japan
  \textbf{49} (2006), no.~2, 144--150.

\bibitem{Webb:1999}
William~A Webb, \emph{An algorithm for super envy-free cake division}, Journal
  of Mathematical Analysis and Applications \textbf{239} (1999), no.~1,
  175--179.

\end{thebibliography}
\end{document}